\documentclass[leqno]{amsart}
\usepackage{amsmath}
\usepackage{amsfonts}
\usepackage{amssymb}
\usepackage{pdfsync}
\usepackage{color,graphicx,tikz-cd}
\usepackage{a4wide}
\usepackage{amscd,amsthm,latexsym}
\usepackage{hyperref}
\usepackage{kotex}
\usepackage{relsize}
\usepackage{tikz}
\usepackage[misc]{ifsym}
\numberwithin{equation}{section}
\usetikzlibrary{snakes}
\usepackage{cite}
\usetikzlibrary{snakes}

\allowdisplaybreaks %

\title[Standard barely set-valued tableaux]{Enumeration of standard barely set-valued tableaux of shifted shapes}

\author{Jang Soo Kim}
\address{
Department of Mathematics, Sungkyunkwan University, Suwon 16420,
South Korea}
\email{jangsookim@skku.edu}

\author{Michael J.\ Schlosser}
\address{Fakult\"at f\"ur Mathematik, Universit\"at Wien,
Oskar-Morgenstern-Platz~1, A-1090 Vienna, Austria}
\email{michael.schlosser@univie.ac.at}

\author{Meesue Yoo}
\address{
Department of Mathematics, Chungbuk National University, Cheongju 28644,
South Korea}
\email{meesueyoo@chungbuk.ac.kr (\Letter)}

\thanks{The first author was supported by NRF grants \#2016R1D1A1A09917506 and \#2016R1A5A1008055.
The second author was partially supported by FWF (Austrian Science Fund) grant P32305.
The third author was supported by Chungbuk National University.}



\date{\today}

\newtheorem{thm}{Theorem}[section]
\newtheorem{lem}[thm]{Lemma}
\newtheorem{prop}[thm]{Proposition}

\theoremstyle{definition}
\newtheorem{exam}[thm]{Example}
\newtheorem{defn}[thm]{Definition}
\newtheorem{conj}[thm]{Conjecture}
\newtheorem{remark}[thm]{Remark}
\newcommand\wt{\operatorname{wt}}
\newcommand\BB{\mathcal{B}}
\newcommand\ddeg{\operatorname{ddeg}}
\newcommand\NE{\operatorname{NE}}
\newcommand\EE{\mathbb{E}}
\newcommand\ZZ{\mathbb{Z}}

\newcommand\rdiag{\operatorname{rdiag}}

\newcommand\SSYT{\operatorname{SSYT}}

\newcommand\SYT{\operatorname{SYT}}
\newcommand\Par{\mathrm{Par}}

\newcommand\la{\lambda}
\newcommand\ma{\mathsf a}
\newcommand\aqddeg{\operatorname{ddeg_{\mathsf a;q}}}

\begin{document}
%----------------------------------------------------------------------------------------------------------

\begin{abstract}
  A standard barely set-valued tableau of shape $\lambda$ is a filling of the
  Young diagram $\lambda$ with integers $1,2,\dots,|\lambda|+1$ such that the
  integers are increasing in each row and column, and every cell contains one
  integer except one cell that contains two integers. Counting standard barely
  set-valued tableaux is closely related to the coincidental down-degree
  expectations (CDE) of lower intervals in Young's lattice. Using $q$-integral
  techniques we give a formula for the number of standard barely set-valued
  tableaux of arbitrary shifted shape. We show how it can be used to recover
  two formulas, originally conjectured by Reiner, Tenner and Yong, and proved
  by Hopkins, for numbers of standard barely set valued tableaux of particular
  shifted-balanced shapes. We also prove a conjecture of Reiner, Tenner and Yong
  on the CDE property of the shifted shape $(n,n-2,n-4,\dots,n-2k+2)$.
  Finally, in the appendix we raise a conjecture on an $\mathsf a;q$-analogue
  of the down-degree expectation with respect to the uniform distribution for
  a specific class of lower intervals in Young's lattice.
\end{abstract}

\maketitle

% \tableofcontents
%----------------------------------------------------------------------------------------------------------

\section{Introduction}

%----------------------------------------------------------------------------------------------------------

Recently, Reiner, Tenner and Yong \cite{RTY2018} introduced a property on
posets, called the coincidental down-degree expectations (CDE). They showed that
many interesting posets have this property. For example, disjoint unions of
chains, Cartesian products of chains, weak Bruhat order on a finite Coxeter
group, Tamari lattices on polygon triangulations, and connected minuscule posets
have the CDE property. Another important family of posets included in their results is a
family of lower intervals of Young's lattice. They also considered lower
intervals of the shifted Young's lattice and proposed two conjectures on the CDE
property of a certain family of lower intervals of the shifted Young's lattice.
One of the two conjectures was proved by Hopkins \cite{Hopkins2017}. The main
goal of this paper is to prove the remaining conjecture.

If $P$ is an interval of (shifted or usual) Young's lattice, the CDE property of
$P$ is closely related to standard barely set-valued tableaux, which are the
main object of interest in this paper. Reiner, Tenner and Yong \cite{RTY2018}
found a formula for the number of barely set-valued tableaux of any regular
shape using an ``uncrowding algorithm'', which is in the spirit of the
Robinson--Schensted algorithm. In this paper we give an analogous formula for
the number of standard barely set-valued tableaux of any shifted shape. In doing
so, we use a modification of their uncrowding algorithm and also the
$q$-integral techniques developed in \cite{KimStanton17}.

We now give precise definitions needed to state our results. Let $P$ be a finite poset.
The \emph{down-degree} $\ddeg(x)$ of an element $x\in P$ is the number of
elements in $P$ covered by $x$. Define $X$ (resp.\ $Y$) to be the random variable
computing the down-degree of $x\in P$ with respect to the uniform distribution
(resp.\ the probability distribution proportional to the number of maximal chains
containing $x$). We say that the poset $P$ has the \emph{coincidental
  down-degree expectations (CDE) property} if $\mathbb{E}(X)=\mathbb{E}(Y)$. For
example if $P$ is the poset in Figure~\ref{fig:CDE}, then $\ddeg(a)=0$,
$\ddeg(b)=1$, $\ddeg(c)=1$, $\ddeg(d)=2$, and the number of maximal chains
through $a,b,c,d$ are $2,1,1,2$, respectively. Thus
\begin{align*}
\EE(X)&=\frac{1}{4}(0+1+1+2)=1,  \\
\EE(Y)&=\frac{1}{2+1+1+2}(0\cdot 2+1\cdot 1+1\cdot 1+2\cdot 2)=1,
\end{align*}
and therefore this poset satisfies the CDE property.

\begin{figure}
\centering
\begin{tikzpicture}[scale=.55]
\draw[thick] (1,0)--(2,1)--(1,2)--(0,1)--cycle;
\filldraw [black] (1,0) circle (3.5pt)
			(2,1) circle (3.5pt)
			(1,2) circle (3.5pt)
			(0,1) circle (3.5pt);
\node[] at (1, -.5) {$a$};
\node[] at (-.5, 1) {$b$};
\node[] at (2.5, 1) {$c$};
\node[] at (1, 2.6) {$d$};
\end{tikzpicture}
  \caption{A poset satisfying the CDE property with $\EE(X)=\EE(Y)=1$. }
  \label{fig:CDE}
\end{figure}
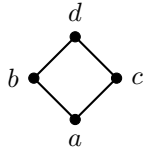

A \emph{partition} is a weakly decreasing sequence
$\lambda=(\lambda_1,\lambda_2,\dots,\lambda_k)$ of positive integers. Each
integer $\lambda_i$ is called a \emph{part} of $\lambda$. The \emph{size}
$|\lambda|$ of $\lambda$ is the sum of its parts and the \emph{length}
$\ell(\lambda)$ of $\lambda$ is the number of its parts. If all parts in
$\lambda$ are distinct, we say that $\lambda$ is \emph{strict}. The set of all
partitions is denoted by $\Par$ and the set of all partitions with at most $n$
parts by $\Par_n$. We also denote by $\Par^*$ the set of all strict partitions
and by $\Par^*_n$ the set of all strict partitions with at most $n$ parts.

If a partition $\lambda$ has $\ell(\lambda)=n$ parts, then we will use the
convention that $\lambda_i=0$ for all $i>n$. For two partitions $\lambda$ and
$\mu$, define
\[
\lambda+\mu=(\lambda_1+\mu_1,\lambda_2+\mu_2,\dots,\lambda_N+\mu_N),
\]
where $N=\max(\ell(\lambda),\ell(\mu))$.

The \emph{Young diagram} of a partition $\lambda=(\lambda_1,\dots,\lambda_k)$ is
the top-left justified array of unit squares (or cells) in which the $i$th row has
$\lambda_i$ squares. If $\lambda$ is strict, the \emph{shifted Young diagram} of
$\lambda$ is the array obtained from the Young diagram of $\lambda$ by shifting
the $i$th row to the right by $i-1$ units for each $1\le i\le \ell(\lambda)$.
See Figure~\ref{fig:YD}. We will identify a partition with its Young diagram (or
its shifted Young diagram if it is strict).

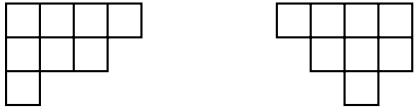
\begin{figure}
  \centering
  \begin{tikzpicture}[scale=.45]
\draw[thick] (0,0)--(1,0)--(1,1)--(3,1)--(3,2)--(4,2)--(4,3)--(0,3)--(0,0)--cycle;
\draw[thick] (0,1)--(1,1)
		(0,2)--(3,2)
		(1,1)--(1,3)
		(2,1)--(2,3)
		(3,2)--(3,3);
\draw[thick] (8,3)--(8,2)--(9,2)--(9,1)--(10,1)--(10,0)--(11,0)--(11,1)--(12,1)--(12,3)--(8,3)--cycle
		(9,2)--(12,2)
		(10,1)--(11,1)
		(9,2)--(9,3)
		(10,1)--(10,3)
		(11,1)--(11,3);
\end{tikzpicture}
  \caption{The Young diagram (left) and the shifted Young diagram (right) of
    $\lambda=(4,3,1)$.}
  \label{fig:YD}
\end{figure}

\emph{Young's lattice} is the poset in which the elements are all partitions and
the relations are given by the inclusion of their Young diagrams, i.e., $\mu\le
\lambda$ if and only if the Young diagram of $\mu$ is contained in that of
$\lambda$. Similarly, the \emph{shifted Young's lattice} is the poset in which
the elements are all strict partitions and the relations are given by the
inclusion of their shifted Young diagrams. In this paper we are interested in
whether a given lower interval $[\emptyset, \lambda]$ in Young's lattice
or in the shifted Young's lattice has the CDE property. 

The CDE property of a poset was first observed by Chan, L\'opez Mart\'in,
Pflueger and Teixidor i Bigas \cite[Remark~2.17]{CMP2018}. They showed that the
lower interval $[\emptyset,\lambda]$ in Young's lattice for a rectangular shape
$\lambda=(a^b)=(a,a,\dots,a)$ with $b$ rows of length $a$ has the CDE property
and the expectations are given by $\mathbb{E}(X)=\mathbb{E}(Y)=\frac{ab}{a+b}$.
This result played an important role when they reproved a formula for the genera
of Brill--Noether curves due to Eisenbud--Harris \cite{EH1987} and Pirola
\cite{Pirola1985}. Chan, Haddadan, Hopkins and Moci \cite{CHHM2017} generalized
the CDE property of a rectangular shape to a much broader class of partitions,
which we now describe.

Let $\mu$ be a Young diagram. An \emph{inner corner} of $\mu$ is a cell
$x\not\in \mu$ such that $\mu\cup\{x\}$ is a Young diagram. We say that $\mu$ is
\emph{balanced} if the left cell of every inner corner of $\mu$ lies on
the line connecting the top right corner of the last cell in the first row of
$\mu$ and the bottom left corner of the last cell in the first column of $\mu$.
If in addition the line has slope $m$, then $\mu$ is \emph{balanced of slope
  $m$}. See Figure~\ref{fig:balanced} for an example of a balanced Young
diagram.

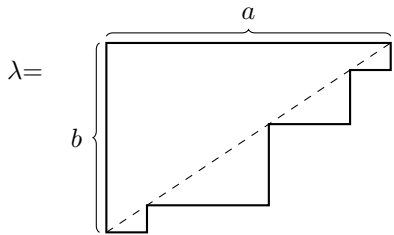
\begin{figure}
$$
    \begin{tikzpicture}[scale=.18]
    \draw[thick] (0,0)--(3,0)--(3,2)--(12,2)--(12,8)--(18,8)--(18,12)--(21,12)--(21,14)--(0,14)--(0,0)--cycle;
    \draw[dashed] (0,0)--(21,14);
    \draw [snake=brace] (-.5,0)--(-.5,14)
    				(0, 14.5)--(21,14.5);
   \node[] at (-2.2, 7) {$b$};
   \node[] at (10.5, 16.2) {$a$};
   \node[] at (-6, 12) {$\lambda$=};
    \end{tikzpicture}$$
  \caption{A typical shape of a balanced Young diagram of slope $b/a$.}
  \label{fig:balanced}
\end{figure}

\begin{thm}\cite[Corollary~3.8]{CHHM2017}
  \label{thm:CHHM}
  Let $\lambda$ be a balanced Young diagram of slope $m$. Then
  $[\emptyset,\lambda]$ has the CDE property and
  \[
\mathbb{E}(X) = \mathbb{E}(Y) = \frac{\lambda_1}{1+m^{-1}}.
  \]
\end{thm}
We note that Chan, Haddadan, Hopkins and Moci \cite{CHHM2017} in fact proved the
CDE property of more general balanced skew shapes. Related further results were
achieved by Hopkins~\cite{Hopkins2019} who in particular proved a conjecture by
Reiner, Tenner and Yong \cite[Conjecture~1.2]{RTY2018} on the CDE property
for vexillary words of a particular shape.

Interestingly, as Hopkins \cite[Remark 3.7]{Hopkins2017} observed and checked
for all partitions of size at most $30$, it seems that the converse of
Theorem~\ref{thm:CHHM} is also true.

\begin{conj}\label{conj:Hopkins1}
  Let $\lambda$ be a Young diagram. Then $[\emptyset,\lambda]$ has the CDE
  property if and only if $\lambda$ is balanced (of any slope).
\end{conj}

Reiner, Tenner and Yong \cite{RTY2018} also considered the CDE property of a
lower interval of the shifted Young's lattice. In what follows we describe their
conjectures on certain shifted Young diagrams.

Let $\lambda=(\lambda_1,\dots,\lambda_n)$ be a strict partition with $n$ parts.
Define $\mu$ to be the Young diagram given by
\[
  \mu=
\begin{cases}
\lambda-\delta_{n+1} = (\lambda_1-n, \lambda_2-n+1,\dots,\lambda_n-1),& \mbox{if $\lambda_n=1$,}\\
\lambda-\delta_{n} =(\lambda_1-n+1, \lambda_2-n+2,\dots,\lambda_n),& \mbox{if $\lambda_n>1$,}
\end{cases}
\]
where $\delta_n=(n-1,n-2,\dots,1,0)$. Note that if $\lambda_n=1$, then
$\ell(\mu)<n$. We say that the shifted Young diagram of $\lambda$ is \emph{balanced} if $\mu$ is a balanced
Young diagram of slope $1$. See Figure~\ref{fig:shifted_balanced} for examples
of balanced shifted Young diagrams.

\begin{figure}[ht]
\centering
\begin{tikzpicture}[scale=.5]
\draw[fill=gray!40] (4,2) rectangle (5,4);
\draw[fill=gray!40] (5,3) rectangle (6,4);
\draw (0,4)--(0,3)--(1,3)--(1,2)--(2,2)--(2,1)--(3,1)--(3,0)--(4,0)--(4,2)--(5,2)--(5,3)--(6,3)--(6,4)--(0,4)--cycle;
\draw (1,3)--(5,3)
	(2,2)--(4,2)
	(3,1)--(4,1)
	(1,3)--(1,4)
	(2,2)--(2,4)
	(3,1)--(3,4)
	(4,2)--(4,4)
	(5,3)--(5,4);
\draw[dashed] (4,2)--(6,4);
\end{tikzpicture}\qquad\qquad
\begin{tikzpicture}[scale=.5]
\draw[fill=gray!40] (3,0) rectangle (6,4);
\draw[fill=gray!40] (6,3) rectangle (7,4);
\draw (0,4)--(0,3)--(1,3)--(1,2)--(2,2)--(2,1)--(3,1)--(3,0)--(6,0)--(6,3)--(7,3)--(7,4)--(0,4)--cycle;
\draw (1,3)--(6,3)
	(2,2)--(6,2)
	(3,1)--(6,1)
	(1,3)--(1,4)
	(2,2)--(2,4)
	(3,1)--(3,4)
	(4,0)--(4,4)
	(5,0)--(5,4)
	(6,3)--(6,4);
\draw[dashed] (3,0)--(7,4);
\end{tikzpicture}
  \caption{Examples of balanced shifted Young diagrams. On the left
    $\mu=\lambda-\delta_5=(2,1)$ and on right 
    $\mu=\lambda-\delta_4=(4,3,3,3)$.}
  \label{fig:shifted_balanced}
\end{figure}
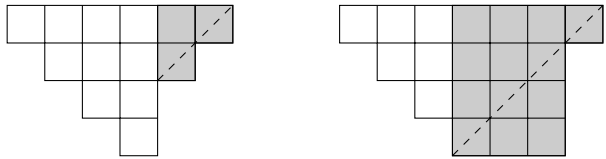

Reiner, Tenner and Yong conjectured the following CDE property for a special
class of strict partitions. Here, $\delta_e (a^a)$ is the Young diagram
obtained from that of $\delta_e$ by replacing each cell by the $a\times a$
square $(a^a)$, see Figure~\ref{fig:circ}. 

\begin{figure}
\begin{tikzpicture}[scale=.43]
\draw[fill=gray!40] (8,2) rectangle (10,8);
\draw[fill=gray!40] (10,4) rectangle (12,8);
\draw[fill=gray!40] (12,6) rectangle (14,8);
\draw (0,8)--(14,8)--(14,6)--(12,6)--(12,4)--(10,4)--(10,2)--(8,2)--(8,0)--(7,0)--(7,1)--(6,1)--(6,2)--(5,2)--(5,3)--(4,3)--(4,4)--(3,4)--(3,5)--(2,5)--(2,6)--(1,6)--(1,7)--(0,7)--(0,8)--cycle;
\draw (1,7)--(14,7);
\draw (2,6)--(12,6);
\draw (3,5)--(12,5);
\draw (4,4)--(10,4);
\draw (5,3)--(10,3);
\draw (6,2)--(8,2);
\draw (7,1)--(8,1);
\draw (1,7)--(1,8);
\draw (2,6)--(2,8);
\draw (3,5)--(3,8);
\draw (4,4)--(4,8);
\draw (5,3)--(5,8);
\draw (6,2)--(6,8);
\draw (7,1)--(7,8);
\draw (8,0)--(8,8);
\draw (10,2)--(10,8);
\draw (9,2)--(9,8);
\draw (11,4)--(11,8);
\draw (12,4)--(12,8);
\draw (13,6)--(13,8);
\end{tikzpicture}
\caption{The shifted Young diagram of $\lambda=\delta_9 +\delta_4(2^2)$.}
\label{fig:circ}
\end{figure}
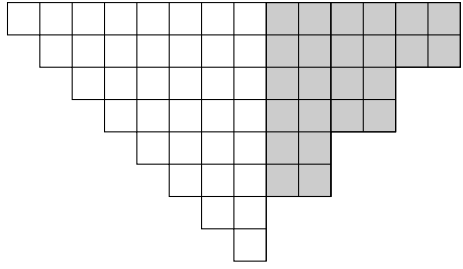

\begin{thm}\cite[Theorem~2.23]{RTY2018}
  \label{thm:RTY1}
  Let $\lambda$ be the shifted shape $\delta_d +\delta_e  (a^a)$ with
  $a,d,e\ge 1$ and $d>a(e-1)+1$. Then the interval $[\emptyset,\lambda]$ has the
  CDE property and
\begin{equation}
  \label{eq:RTY}
\mathbb{E}(X) = \mathbb{E}(Y) = \frac{\lambda_1+1}{4}.
\end{equation}
\end{thm}

Hopkins proved the following generalization of Theorem~\ref{thm:RTY1}.

\begin{thm}\cite[Theorem~4.2]{Hopkins2017}
 \label{thm:shifted-balanced}
 Let $\lambda$ be a balanced shifted Young diagram. Then the interval
 $[\emptyset,\lambda]$ has the CDE property and
  \[
\mathbb{E}(X) = \mathbb{E}(Y) = \frac{\lambda_1+1}{4}.
  \]
\end{thm}

\begin{figure}
\begin{tikzpicture}[scale=.45]
\draw (0,4)--(0,3)--(1,3)--(1,2)--(2,2)--(2,1)--(3,1)--(3,0)--(7,0)--(7,1)--(8,1)--(8,2)--(9,2)--(9,3)--(10,3)--(10,4)--(0,4)--cycle;
\draw (1,3)--(9,3)
	(2,2)--(8,2)
	(3,1)--(7,1)
	(1,3)--(1,4)
	(2,2)--(2,4)
	(3,1)--(3,4)
	(4,0)--(4,4)
	(5,0)--(5,4)
	(6,0)--(6,4)
	(7,1)--(7,4)
	(8,2)--(8,4)
	(9,3)--(9,4);
\end{tikzpicture}
\caption{A trapezoidal shifted Young diagram with $n=10$, $k=4$.}
  \label{fig:trape}
\end{figure}
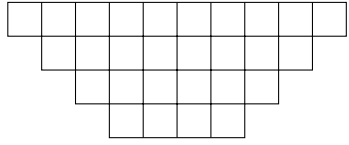

A shifted Young diagram is \emph{trapezoidal} if it is of the form
$(n,n-2,\dots,n-2k+2)$. See Figure~\ref{fig:trape} for an example. The main goal
of this paper is to prove the following theorem conjectured by Reiner, Tenner
and Yong.

\begin{thm}\cite[Conjecture~2.24]{RTY2018}
  \label{thm:trapezoidal}
  Let $\lambda$ be a trapezoidal shifted Young diagram. 
Then the interval
 $[\emptyset,\lambda]$ has the CDE property and
  \[
\mathbb{E}(X)=\mathbb{E}(Y) = \frac{|\lambda|}{\lambda_1+1}.
  \]
\end{thm}

Interestingly, as before, it seems that the shifted Young diagrams described in
Theorems~\ref{thm:shifted-balanced} and \ref{thm:trapezoidal} are the only ones
having the CDE property. Hopkins \cite[Remark 4.9]{Hopkins2017} suggested the
following conjecture and confirmed it for all shifted Young diagrams of size at
most $18$. Using our method of computing $\mathbb{E}(X)$ and $\mathbb{E}(Y)$, we
have checked the conjecture for all shifted Young diagrams of size up to $50$.

\begin{conj}
  Let $\lambda$ be a shifted Young diagram. Then the interval
  $[\emptyset,\lambda]$ has the CDE property if and only if $\lambda$ is
  balanced or trapezoidal.
\end{conj}

There are stronger properties on posets, called mCDE and tCDE. The mCDE property
concerns multichains and the tCDE property concerns toggles on posets introduced
by Striker and Williams \cite{StrikerWilliams2012}. If the poset is graded, then
the mCDE property implies the CDE property. If the poset is distributive, then
the tCDE property implies both the mCDE and CDE properties. See
\cite{Hopkins2017} for more details. Hopkins \cite{Hopkins2017} showed
Theorem~\ref{thm:shifted-balanced} by showing that the poset in the theorem has
the tCDE property. However, his proof cannot apply to
Theorem~\ref{thm:trapezoidal} since the poset in this theorem does not have the
tCDE property, see \cite[Section~4.2]{Hopkins2017}.

Recall that our main goal is to prove Theorem~\ref{thm:trapezoidal}. Our
approach is to compute $\EE(X)$ and $\EE(Y)$ separately and show that they are equal. If $\lambda$ is a
(shifted) Young diagram, computing $\EE(Y)$ for the lower interval
$[\emptyset,\lambda]$ is closely related to standard barely set-valued tableaux
of shape $\lambda$, which are the main object of focus in this paper.

Let $\lambda$ be a partition. A \emph{standard Young tableau} of shape $\lambda$
is a filling of $\lambda$ with integers $1,2,\dots,|\lambda|$ in such a way that
the integers are increasing in each row and each column. See
Figure~\ref{fig:SBY} for an example. The number of standard
Young tableaux of shape $\lambda$ is denoted by $f^\lambda$. The famous
hook-length formula \cite{Frame1954} states that if $\lambda$ is a partition of
length $n$, then
\begin{equation}
  \label{eq:HLF1}
f^\lambda=  \frac{|\lambda|!}{\prod_{u\in \lambda}h_\lambda(u)}
=\dfrac{|\lambda|!}{\prod_{i=1}^n(\lambda_i+n-i) !} \prod_{1\le i<j\le n}
(\lambda_i -\lambda_j-i+j),
\end{equation}
where for each cell $u=(i,j)\in\lambda$, the \emph{hook length} $h_\lambda(u)$
is defined to be $\lambda_i+\lambda'_j-i-j+1$ and $\lambda'_j$ is the number
of integers $k$ such that $\lambda_k\ge j$.

\begin{figure}
 \begin{tikzpicture}[scale=.45]
\draw (0,0)--(1,0)--(1,1)--(3,1)--(3,2)--(4,2)--(4,3)--(0,3)--(0,0)--cycle;
\draw (0,1)--(1,1)
		(0,2)--(3,2)
		(1,1)--(1,3)
		(2,1)--(2,3)
		(3,2)--(3,3);
\node[] at (.5, 2.5) {$1$};
\node[] at (1.5, 2.5) {$2$};
\node[] at (2.5, 2.5) {$5$};
\node[] at (3.5, 2.5) {$8$};
\node[] at (.5, 1.5) {$3$};
\node[] at (1.5, 1.5) {$6$};
\node[] at (2.5, 1.5) {$7$};
\node[] at (.5, .5) {$4$};
\end{tikzpicture}\qquad\qquad
 \begin{tikzpicture}[scale=.45]
\draw (0,0)--(1,0)--(1,1)--(3,1)--(3,2)--(4,2)--(4,3)--(0,3)--(0,0)--cycle;
\draw (0,1)--(1,1)
		(0,2)--(3,2)
		(1,1)--(1,3)
		(2,1)--(2,3)
		(3,2)--(3,3);
\node[] at (.5, 2.5) {$1$};
\node[] at (1.5, 2.5) {$2$};
\node[] at (2.5, 2.5) {$3$};
\node[] at (3.5, 2.5) {$6$};
\node[] at (.5, 1.5) {$4$};
\node[] at (1.3, 1.65) {$5$};
\node[] at (1.75, 1.35) {$7$};
\node[] at (2.5, 1.5) {$9$};
\node[] at (.5, .5) {$8$};
\end{tikzpicture}
  \caption{A standard Young tableau of shape $(4,3,1)$ (left) and a standard
    barely set-valued tableau of shape $(4,3,1)$ (right).}
  \label{fig:SBY}
\end{figure}

A \emph{standard barely set-valued tableau} of shape $\lambda$ is a filling of
the Young diagram $\lambda$ with integers $1,2,\dots,|\lambda|+1$ such that the
integers are increasing in each row and column, and every cell contains one
integer except one cell that contains two integers. See
Figure~\ref{fig:SBY} for an example. Let $f^\lambda(+1)$ denote
the number of standard barely set-valued tableaux of shape $\lambda$. The number
$f^\lambda(+1)$ is also known to be equal to a certain coefficient of the
Grothendieck polynomial due to Buch \cite[Theorem~3.1]{Buch2002}.

Reiner, Tenner and Yong \cite[Corollary~3.7]{RTY2018} showed that if the poset
is the lower interval $[\emptyset,\lambda]$ for a partition $\lambda$, then
\begin{equation}\label{eq:eyf}
\mathbb{E}(Y) = \frac{f^\lambda(+1)}{(|\lambda|+1)f^\lambda}.
\end{equation}
Since we already have a formula for $f^\lambda$, in order to evaluate
$\mathbb{E}(Y)$ it suffices to find $f^\lambda(+1)$. Using what they call an
uncrowding algorithm, Reiner, Tenner and
Yong \cite[Corollary~3.11 and Remark~3.13]{RTY2018} showed that
\begin{equation}
  \label{eq:RTY f+1}
f^\lambda(+1)
= \sum_{k:\lambda_k<\lambda_{k-1}} \lambda_k f^{\lambda\cup\{(k,\lambda_k+1)\}},
\end{equation}
where we set $\lambda_0=\infty$. They used \eqref{eq:RTY f+1} to evaluate
$\EE(Y)$ for the interval $[\emptyset,\lambda]$ when $\lambda$ is the
rectangular staircase shape $\delta_d(b^a)$, which is the Young diagram obtained
from that of $\delta_d$ by replacing each cell by the $a\times b$ rectangle
$(b^a)$. Using \eqref{eq:HLF1}, one can restate \eqref{eq:RTY f+1} as follows.

\begin{thm}\label{thm:f+1}
  Let $\lambda$ be a partition with $n$ parts. Then
  \begin{align*}
    f^\lambda(+1) =\sum_{k:\lambda_k<\lambda_{k-1}}&
    \frac{\lambda_k(|\lambda|+1)!}{(\lambda_k+n-k+1)\prod_{i=1}^n(\lambda_i+n-i)!}\\
&\times \prod_{1\le i<j\le n}(\lambda_i-\lambda_j-i+j)
\prod_{i\ne k} \frac{\lambda_k-\lambda_i+k-i-1}{\lambda_k-\lambda_i+k-i},
  \end{align*}
where $\lambda_0=\infty$. Equivalently, the expectation $\EE(Y)$ for the
interval $[\emptyset,\lambda]$ is given by
\begin{equation}
  \label{eq:f+1}
\EE(Y)= \frac{f^\lambda(+1)}{(|\lambda|+1)f^\lambda} =  \sum_{k:\lambda_k<\lambda_{k-1}}
\frac{\lambda_k}{\lambda_k+n-k+1}
\prod_{i\ne k} \frac{\lambda_k-\lambda_i+k-i-1}{\lambda_k-\lambda_i+k-i}.
\end{equation}
\end{thm}

\begin{figure}
\begin{tikzpicture}[scale=.45]
\draw (0,3)--(0,2)--(1,2)--(1,1)--(2,1)--(2,0)--(3,0)--(3,1)--(4,1)--(4,3)--(0,3)--cycle;
\draw (1,2)--(4,2)
	(2,1)--(3,1)
	(1,2)--(1,3)
	(2,1)--(2,3)
	(3,1)--(3,3);
\node[] at (.5, 2.5) {$1$};
\node[] at (1.5, 2.5) {$2$};
\node[] at (2.5, 2.5) {$4$};
\node[] at (3.5, 2.5) {$5$};
\node[] at (1.5, 1.5) {$3$};
\node[] at (2.5, 1.5) {$6$};
\node[] at (3.5, 1.5) {$8$};
\node[] at (2.5, .5) {$7$};
\end{tikzpicture}\qquad\qquad
\begin{tikzpicture}[scale=.45]
\draw (0,3)--(0,2)--(1,2)--(1,1)--(2,1)--(2,0)--(3,0)--(3,1)--(4,1)--(4,3)--(0,3)--cycle;
\draw (1,2)--(4,2)
	(2,1)--(3,1)
	(1,2)--(1,3)
	(2,1)--(2,3)
	(3,1)--(3,3);
\node[] at (.5, 2.5) {$1$};
\node[] at (1.5, 2.5) {$2$};
\node[] at (2.5, 2.5) {$4$};
\node[] at (3.5, 2.5) {$6$};
\node[] at (1.5, 1.5) {$3$};
\node[] at (2.3, 1.65) {$5$};
\node[] at (2.75, 1.35) {$7$};
\node[] at (3.5, 1.5) {$9$};
\node[] at (2.5, .5) {$8$};
\end{tikzpicture}
  \caption{A standard Young tableau of shifted shape $(4,3,1)$ (left) and a
    standard barely set-valued tableau of shifted shape $(4,3,1)$ (right).}
  \label{fig:shifted_SYT}
\end{figure}

Now let $\lambda$ be a strict partition. A \emph{standard Young tableau} and a
\emph{standard barely set-valued tableau} of shifted shape $\lambda$ are defined
similarly as fillings of the shifted Young diagram $\lambda$, see
Figure~\ref{fig:shifted_SYT}. We denote by $g^\lambda$ (resp.~$g^\lambda(+1)$)
the number of standard Young tableaux (resp.~standard barely set-valued
tableaux) of shifted shape $\lambda$.

For a strict partition $\lambda$ with $n$ parts, there is a shifted hook-length
formula due to Thrall \cite{Thrall1952}:
\begin{equation}
\label{eqn:Thrall}
g^\lambda =  \frac{|\lambda|!}{\prod_{u\in \lambda}h_{\lambda}(u)}
=\dfrac{|\lambda|!}{\prod_{i=1}^n\lambda_i !} \prod_{1\le i<j\le n} \dfrac{\lambda_i -\lambda_j}{\lambda_i + \lambda_j}.
\end{equation}
Since we do not need the shifted hook length $h_\lambda(u)$ in this paper,
we refer the reader  to \cite[Fig.~7]{KimYoo_skew} for its definition.

If $\lambda$ is a shifted Young diagram, the expectation $\EE(Y)$ for the
interval $[\emptyset,\lambda]$ can be similarly
computed\footnote{The proof is completely analogous to that
of \eqref{eq:eyf} for ordinary partitions by
Reiner, Tenner and Yong \cite[Corollary~3.7]{RTY2018}, so we omit the details.}:
\begin{equation}
\mathbb{E}(Y) = \frac{g^\lambda(+1)}{(|\lambda|+1)g^\lambda}.
\end{equation}

In this paper, using a modification of the uncrowding algorithm and some
$q$-integral techniques from \cite{KimStanton17}, we express $g^\lambda(+1)$ as a
sum of $g^\mu$'s in a similar fashion as \eqref{eq:RTY f+1}. This leads us to the
following explicit formula for $g^\lambda(+1)$, which is a shifted analogue of
Theorem~\ref{thm:f+1}.

\begin{thm}\label{thm:main}
  Let $\lambda$ be a shifted Young diagram with $n$ rows. Then
  \begin{align*}
g^\lambda(+1)
={}&\frac{(|\lambda|+1)!}{2\prod_{i=1}^n\lambda_i!}
\prod_{1\le i<j\le n} \dfrac{\lambda_i -\lambda_j}{\lambda_i + \lambda_j}\\
&\times \!
\left(  n +
\sum_{k:\lambda_k\le \lambda_{k-1}-2}
\frac{\lambda_k-2n+2k-1}{\lambda_k+1}
\prod_{i\ne k} \frac{\lambda_k+\lambda_i}{\lambda_k-\lambda_i}
\frac{\lambda_k-\lambda_i+1}{\lambda_k+\lambda_i+1}
 \right),
  \end{align*}
where $\lambda_0=\infty$. Equivalently, the expectation $\EE(Y)$ for the
interval $[\emptyset,\lambda]$ is given by
\begin{equation}
  \label{eq:g+1}
\EE(Y)= \frac{g^\lambda(+1)}{(|\lambda|+1)g^\lambda} 
= \frac{1}{2}\!
\left(  n +
\sum_{k:\lambda_k\le \lambda_{k-1}-2}
\frac{\lambda_k-2n+2k-1}{\lambda_k+1}
\prod_{i\ne k} \frac{\lambda_k+\lambda_i}{\lambda_k-\lambda_i}
\frac{\lambda_k-\lambda_i+1}{\lambda_k+\lambda_i+1}
 \right).
\end{equation}
\end{thm}

The advantage of Theorem~\ref{thm:main} is that it expresses, for any
shifted Young diagram $\lambda$, the expectation $\EE(Y)$ for the interval
$[\emptyset,\lambda]$ as an explicit sum.
In some fortunate cases, we are able to simplify the sum using the
theory of hypergeometric series and express $\EE(Y)$
as a product. We apply exactly this procedure to
give new proofs of Theorems~\ref{thm:RTY1} and
\ref{thm:shifted-balanced} and a first proof of Theorem~\ref{thm:trapezoidal}.

The remainder of this paper is organized as follows. In
Section~\ref{sec:preliminaries} we give basic definitions and results on
semistandard Young tableaux and $q$-integrals. In Section~\ref{sec:proof} we
prove Theorem~\ref{thm:main}. In Section~\ref{sec:enum-stand-barely} we use
Theorem~\ref{thm:main} to evaluate $\EE(Y)$ for the intervals described in
Theorems~\ref{thm:RTY1}, \ref{thm:shifted-balanced}, and \ref{thm:trapezoidal}.
In Section~\ref{sec:sum-down-degrees} we compute $\EE(X)$ for the interval
described in Theorem~\ref{thm:trapezoidal}, hence complete the proof of
Theorem~\ref{thm:trapezoidal}. Finally, in Appendix~\ref{sec:an-ma-q} we propose a
$q$-analogue, with an extra parameter $\ma$, for the down-degree expectation
$\EE(X)$ for which we conjecture a product formula for lower intervals
$[\emptyset,\la]$ of Young's lattice, in case $\la$ is an ordinary partition
 of balanced shape of any slope.

%----------------------------------------------------------------------------------------------------------

\section{Preliminaries}
\label{sec:preliminaries}

%----------------------------------------------------------------------------------------------------------

In this section we review some basic definitions and results relating
$q$-integrals and semistandard Young tableaux.
We use the standard notation in $q$-series:
\[
  (a;q)_n = (1-a)(1-aq)\cdots (1-aq^{n-1}).
\]

Let $\mu$ be a partition and $\lambda$ a strict partition. One can naturally
identify the Young diagram of $\mu$ with the set $\{(i,j)\in\ZZ\times\ZZ: 1\le
i\le l(\mu), \,1\le j\le \mu_i\}$ and the shifted Young diagram of $\lambda$
with the set $\{(i,j)\in\ZZ\times\ZZ: 1\le i\le l(\lambda),\, i\le j\le
\lambda_i+i-1\}$. In this section and the next we consider another family of
diagrams defined as follows. An \emph{extended shifted Young diagram} is a
diagram obtained from a shifted Young diagram $\lambda$ by adding a cell below
the main diagonal, i.e., a cell in row $i$ and column $i-1$ for some $i\ge1$. In
this case we denote the extended shifted Young diagram by
$\lambda\cup\{(i,i-1)\}$. See Figure~\ref{fig:extended} for an example.

\begin{figure}
\begin{tikzpicture}[scale=.5]
\draw[fill=gray!40] (1,1) rectangle (2,2);
\draw[thick] (1,1)--(2,1)--(2,2)--(1,2)--(1,1)--cycle;
\draw (0,4)--(0,3)--(1,3)--(1,2)--(2,2)--(2,1)--(3,1)--(3,0)--(6,0)--(6,1)--(7,1)--(7,3)--(8,3)--(8,4)--(0,4)--cycle;
\draw (1,3)--(7,3)
	(2,2)--(7,2)
	(3,1)--(6,1)
	(1,3)--(1,4)
	(2,2)--(2,4)
	(3,1)--(3,4)
	(4,0)--(4,4)
	(5,0)--(5,4)
	(6,1)--(6,4)
	(7,3)--(7,4);
\end{tikzpicture}\qquad\qquad
\begin{tikzpicture}[scale=.5]
\draw (1,1)--(2,1)--(2,2)--(1,2)--(1,1)--cycle;
\draw (0,4)--(0,3)--(1,3)--(1,2)--(2,2)--(2,1)--(3,1)--(3,0)--(6,0)--(6,1)--(7,1)--(7,3)--(8,3)--(8,4)--(0,4)--cycle;
\draw (1,3)--(7,3)
	(2,2)--(7,2)
	(3,1)--(6,1)
	(1,3)--(1,4)
	(2,2)--(2,4)
	(3,1)--(3,4)
	(4,0)--(4,4)
	(5,0)--(5,4)
	(6,1)--(6,4)
	(7,3)--(7,4);
\node[] at (.5, 3.5) {$0$};
\node[] at (1.5, 3.5) {$0$};
\node[] at (2.5, 3.5) {$0$};
\node[] at (3.5, 3.5) {$0$};
\node[] at (4.5, 3.5) {$1$};
\node[] at (5.5, 3.5) {$1$};
\node[] at (6.5, 3.5) {$1$};
\node[] at (7.5, 3.5) {$2$};
\node[] at (1.5, 2.5) {$1$};
\node[] at (2.5, 2.5) {$1$};
\node[] at (3.5, 2.5) {$2$};
\node[] at (4.5, 2.5) {$2$};
\node[] at (5.5, 2.5) {$2$};
\node[] at (6.5, 2.5) {$3$};
\node[] at (1.5, 1.5) {$2$};
\node[] at (2.5, 1.5) {$3$};
\node[] at (3.5, 1.5) {$3$};
\node[] at (4.5, 1.5) {$3$};
\node[] at (5.5, 1.5) {$4$};
\node[] at (6.5, 1.5) {$4$};
\node[] at (3.5, .5) {$4$};
\node[] at (4.5, .5) {$5$};
\node[] at (5.5, .5) {$5$};
\end{tikzpicture}
  \caption{An extended shifted Young diagram $\lambda\cup\{(3,2)\}$ for
    $\lambda=(8,6,5,3)$ on the left and a semistandard Young tableau of this
    shape on the right.}
  \label{fig:extended}
\end{figure}
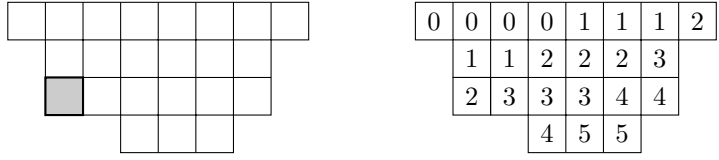

From now on, a diagram means a Young diagram, a shifted Young diagram, or an
extended shifted Young diagram.

Let $\pi$ be a diagram. A \emph{standard Young tableau} of shape $\pi$ is a
filling of the cells in $\pi$ with integers $1,2,\dots,|\pi|$ such that the
integers are increasing in each row and in each column and each integer $1\le
i\le |\pi|$ is used exactly once. A \emph{semistandard Young tableau} of shape
$\pi$ is a filling of $\pi$ with nonnegative integers such that the integers are
weakly increasing in each row and strictly increasing in each column. See
Figure~\ref{fig:extended} for an example. The set of standard
(resp.~semistandard) Young tableaux of shape $\pi$ is denoted by $\SYT(\pi)$
(resp.~$\SSYT(\pi)$). For $T\in\SSYT(\pi)$, let $|T|$ be the sum of integers in
$T$.

The following proposition can be obtained from a well known result in the $P$-partition
theory \cite[Theorem~3.15.7]{EC1}. Note that if $\pi$ is a Young diagram, then
$|\SYT(\pi)| = f^\pi$, and if $\pi$ is a shifted Young diagram, then
$|\SYT(\pi)| = g^\pi$.

\begin{prop}\label{prop:SYT_SSYT}
  For any diagram $\pi$,
\[
  |\SYT(\pi)|=\lim_{q\to1} (q;q)_{|\pi|}\sum_{T\in \SSYT(\pi)}q^{|T|}.
\]
\end{prop}

For $\lambda\in\Par_n$ and a sequence $x=(x_1,\dots,x_n)$ of variables, let
\[
{a}_\lambda(x) = \det(x_{j}^{\lambda_i})_{i,j=1}^n,\qquad
\overline{a}_\lambda(x) = (-1)^{\binom n2}{a}_\lambda(x).
\]
For an $n$-variable function $f(x)=f(x_1,\dots,x_n)$, we write $f(q^\lambda)$ to
mean $f(q^{\lambda_1},\dots,q^{\lambda_n})$.

The following result will be used, see \cite[Theorem 8.7]{KimStanton17}
or \cite[Corollary~2.3]{KimYoo_skew}.
\begin{prop}\label{prop:alternant}
\label{cor:gfs}
Let $\lambda$ be a shifted Young diagram with $n$ parts and let $\nu$ be a
partition with at most $n$ parts. Then
\[
\sum_{\substack{T\in \SSYT(\lambda) \\ \rdiag (T)=\nu}} q^{|T|}
= \frac{q^{|\nu|}}{\prod_{j=1}^n (q;q)_{\lambda_j-1}} \overline{a}_{\lambda-(1^n)}(q^{\nu}),
\]
where $\rdiag(T)$ is the partition obtained by reading the main diagonal entries
in $T$ in reversed order.
\end{prop}

The \emph{$q$-integral} of $f(x)$ over $[a,b]$ is defined by
\[
\int_{a}^b f(x)d_qx=(1-q)\sum_{i\ge0} (f(bq^i)bq^i-f(aq^i)aq^i),
\]
where $0<q<1$ and the sum is assumed to absolutely converge.
If $q$ approaches $1$, the $q$-integral converges to the usual integral:
\begin{equation}
  \label{eq:qint1}
  \lim_{q\to 1}\int_{a}^b f(x)d_qx= \int_{a}^b f(x)dx.
\end{equation}
In this paper the following multivariate $q$-integral will be considered:
\[
\int_{0\le x_1\le \cdots \le x_n\le 1}f(x_1,\dots, x_n)d_q x
=\int_0^1\int_0^{x_n}\int_0^{x_{n-1}}\dots\int_0^{x_2}f(x_1,\dots, x_n)d_q x_1\cdots d_q x_n.
\]
Note that by \eqref{eq:qint1} we have
\[
  \lim_{q\to 1} \int_{0\le x_1\le \cdots \le x_n\le 1}f(x_1,\dots, x_n)d_q x
  =\int_{0\le x_1\le \cdots \le x_n\le 1}f(x_1,\dots, x_n)d x.
\]

\begin{lem}\cite[Lemma 4.3]{KimStanton17}\label{lem:gf}
For a function $f(x_1,\dots, x_n)$ satisfying $f(x_1,\dots, x_n)=0$ if $ x_i =x_j$ for any $i\ne j$,
$$\sum_{\nu\in \Par_n}q^{|\nu|}f(q^{\nu})=\frac{1}{(1-q)^n}\int_{0\le x_1\le \cdots \le x_n\le 1}f(x_1,\dots, x_n)d_q x.$$
\end{lem}

The following two propositions express the number of standard Young tableaux of
shape $\lambda$ as an integral, when $\lambda$ is a shifted Young diagram and
extended shifted Young diagram, respectively.

\begin{prop}\label{prop:SSYT_int}
  Let $\lambda$ be a strict partition with $n$ parts. Then
\[
g^\lambda
=\frac{|\lambda|!}{\prod_{j=1}^n (\lambda_j-1)!}
\int_{0\le x_1\le \cdots \le x_n\le 1}\overline{a}_{\lambda-(1^n)}(x)d x.
\]
\end{prop}
\begin{proof}
  By Proposition~\ref{prop:alternant},
  \[
  \sum_{T\in\SSYT(\lambda)} q^{|T|}
  = \sum_{\nu\in\Par_n}\sum_{\substack{T\in \SSYT(\lambda) \\ \rdiag (T)=\nu}} q^{|T|}
  =\sum_{\nu\in\Par_n} \frac{q^{|\nu|}}{\prod_{j=1}^n (q;q)_{\lambda_j-1}}
    \overline{a}_{\lambda-(1^n)}(q^{\nu}).
    \]
Then by Lemma~\ref{lem:gf},
\[
  \sum_{T\in\SSYT(\lambda)} q^{|T|}
=\frac{(1-q)^{-n}}{\prod_{j=1}^n (q;q)_{\lambda_j-1}}
\int_{0\le x_1\le \cdots \le x_n\le 1}\overline{a}_{\lambda-(1^n)}(x)d_q x.
\]
Multiplying both sides of this equation by $(q;q)_{|\lambda|}$ and taking the
$q\to1$ limit gives
the result by Proposition~\ref{prop:SYT_SSYT}.
\end{proof}

\begin{prop}\label{prop:SSYT_int2}
  Let $\lambda$ be a strict partition with $n$ parts. Then
for $1\le i\le n$,
\[
g^{\lambda\cup\{(n-i+1,n-i)\}}
=\frac{(|\lambda|+1)!}{\prod_{j=1}^n (\lambda_j-1)!}
\int_{0\le x_1\le \cdots \le x_n\le 1}(x_{i+1}-x_i)\overline{a}_{\lambda-(1^n)}(x)dx,
\]
where $x_{n+1}=1$.
\end{prop}
\begin{proof}
  By the same argument as in the proof of Proposition~\ref{prop:SSYT_int},
  it suffices to prove the following identity:
\begin{equation}
  \label{eq:2a}
  \sum_{T\in\SSYT(\lambda\cup\{(n-i+1,n-i)\})} q^{|T|}
=\frac{q(1-q)^{-n-1}}{\prod_{j=1}^n (q;q)_{\lambda_j-1}}
\int_{0\le y_1\le \cdots \le y_n\le 1}(y_{i+1}-y_i)\overline{a}_{\lambda-(1^n)}(y)d_q y,
\end{equation}
where $y_{n+1}=q^{-1}$. First, observe that
\[
  \sum_{T\in\SSYT(\lambda\cup\{(n-i+1,n-i)\})} q^{|T|}
  = \sum_{\nu\in\Par_n}\sum_{\substack{T\in\SSYT(\lambda\cup\{(n-i+1,n-i)\})\\
  \rdiag (T)=\nu}} q^{|T|} .
\]
Let $\nu\in\Par_n$ and $T\in \SSYT(\lambda\cup\{(n-i+1,n-i)\})$ with $\rdiag (T)=\nu$.
If $s$ is the $(n-i+1,n-i)$-entry of $T$, then $s$ can be any integer
satisfying $\nu_{i+1}< s\le \nu_i$, where $\nu_{n+1}=-1$. Therefore,
\begin{align*}
\sum_{\nu\in\Par_n}\sum_{\substack{T\in\SSYT(\lambda\cup\{(n-i+1,n-i)\})\\
  \rdiag (T)=\nu}} q^{|T|}
  &= \sum_{\nu\in\Par_n}\sum_{\substack{T\in \SSYT(\lambda) \\ \rdiag (T)=\nu}} q^{|T|}
(q^{\nu_{i+1}+1}+q^{\nu_{i+1}+2}+\dots+q^{\nu_i}),
\end{align*}
which is, by Proposition~\ref{prop:alternant}, equal to
  \[
    \sum_{\nu\in\Par_n} \frac{q^{|\nu|}}{\prod_{j=1}^n (q;q)_{\lambda_j-1}}
    \overline{a}_{\lambda-(1^n)}(q^{\nu}) \frac{q(q^{\nu_{i+1}}-q^{\nu_{i}})}{1-q}.
  \]
  Applying Lemma~\ref{lem:gf} to the above sum gives the right hand side of
  \eqref{eq:2a} which completes the proof.
\end{proof}

%----------------------------------------------------------------------------------------------------------

\section{A formula for the number of shifted SBTs}
\label{sec:proof}

%----------------------------------------------------------------------------------------------------------

In this section we prove Theorem~\ref{thm:main} in the introduction, which is a
formula for the number of standard barely set-valued tableaux of shifted shape.

Let $\lambda$ be a shifted Young diagram with $n$ rows. For a standard barely
set-valued tableau (in short, SBT) $T$ of shape $\lambda$, there is a unique cell
that contains two integers. We call this cell the \emph{double cell} of $T$. Denote by
$g^\lambda_k(+1)$ the number of standard barely set-valued tableaux of shifted
shape $\lambda$ with the double cell in column $k$. We also define
$g^\lambda_{i,i}(+1)$ to be the number of standard barely set-valued tableaux of
shape $\lambda$ with double cell in column $i$ and in row $i$. Then, by
definition,
\begin{equation}
  g^\lambda(+1)=\sum_{i=1}^{\lambda_1}g^\lambda_i(+1).
\end{equation}

A \emph{northeast corner} of $\lambda\in\Par^*_n$ is a cell $(i,i+\lambda_i)$
such that $1\le i\le n$ and $\lambda_i\le \lambda_{i-1}-2$, where
$\lambda_0=\infty$. In other words, $(i,j)$ is a northeast corner of $\lambda$
if and only if $1\le i\le n$, $(i,j)\not\in \lambda$ and $\lambda\cup\{(i,j)\}$
is a shifted Young diagram. The set of northeast corners of $\lambda$ is denoted
by $\NE(\lambda)$. See Figure~\ref{fig:NE} for an example. Note that in our
definition $(n+1,n+1)\not\in\NE(\lambda)$ even when $\lambda_n\ge2$ and  
$\lambda\cup\{(n+1,n+1)\}$ is a shifted Young diagram.

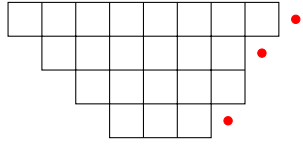
\begin{figure}
\begin{tikzpicture}[scale=.45]
\draw (0,4)--(0,3)--(1,3)--(1,2)--(2,2)--(2,1)--(3,1)--(3,0)--(6,0)--(6,1)--(7,1)--(7,3)--(8,3)--(8,4)--(0,4)--cycle;
\draw (1,3)--(7,3)
	(2,2)--(7,2)
	(3,1)--(6,1)
	(1,3)--(1,4)
	(2,2)--(2,4)
	(3,1)--(3,4)
	(4,0)--(4,4)
	(5,0)--(5,4)
	(6,1)--(6,4)
	(7,3)--(7,4);
\filldraw [red] (8.5,3.5) circle (3.5pt)
                      (7.5, 2.5) circle (3.5pt)
                      (6.5, .5) circle (3.5pt);
\end{tikzpicture}
  \caption{The set $\NE(\lambda)$ of northeast corners of $\lambda=(8,6,5,3)$.}
  \label{fig:NE}
\end{figure}

Since we have the shifted hook length formula for $g^\mu$, the idea of proof of
Theorem~\ref{thm:main} is to express $g^{\lambda}(+1)$ as a sum of $g^\mu$'s
(Proposition~\ref{prop:g+1}). To this end, we need a sequence of lemmas. The
methods used in the proofs of the first two lemmas are similar to the
``uncrowding'' algorithm in \cite{RTY2018}.

\begin{lem}\label{lem:+k2}
For $\lambda\in\Par^*_n$ and $n\le k\le \lambda_1$,
  \[
   g^\lambda_k(+1)= \sum_{(i,j)\in \NE(\lambda),\, j>k}
   g^{\lambda\cup\{(i,j)\}}.
  \]
\end{lem}
\begin{proof}
  It is enough to show that
  \begin{equation}
    \label{eq:1}
    g^\lambda_k(+1) = g^\lambda_{k+1}(+1) + g^{\lambda\cup \{(i+1,k+1)\}},
  \end{equation}
  where $i$ is the length of the $k$th column of $\lambda$. We define
  $g^{\lambda}_{k+1}(+1)=0$ if $k+1>\lambda_1$, and $g^{\lambda\cup
    \{(i+1,k+1)\}}=0$ if $(i+1,k+1)\not\in \NE(\lambda)$.

  We prove \eqref{eq:1} by constructing a bijection from $A$ to $B\cup C$, where
  $A$ is the set of SBT of shape $\lambda$ with double cell in column $k$, $B$
  is the set of SBTs of shape $\lambda$ with double cell in column $k+1$, and
  $C$ is the set of SYTs of shape $\lambda\cup\{(i+1,k+1)\}$.

  Suppose $T\in A$. Let $a<b$ be the integers in the double cell in $T$. Let $c$
  be the smallest integer larger than $b$ in the $(k+1)$st column of $T$. Define
  $T'$ to be the SBT obtained from $T$ by moving $b$ to the cell containing $c$.
  If there is no such integer $c$, then define $T'$ to be the SYT of shape
  $\lambda\cup\{(i+1,k+1)\}$ obtained from $T$ by moving $b$ to the new cell
  $(i+1,k+1)$. See Figure~\ref{fig:uncrowding1} for examples of this map.

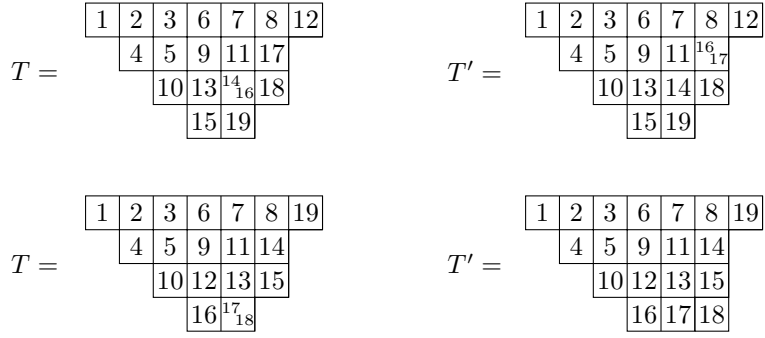
\begin{figure}
$$
\begin{tikzpicture}[scale=.45]
\draw (0,4)--(0,3)--(1,3)--(1,2)--(2,2)--(2,1)--(3,1)--(3,0)--(5,0)--(5,1)--(6,1)--(6,3)--(7,3)--(7,4)--(0,4)--cycle;
\draw (1,3)--(7,3)
	(2,2)--(6,2)
	(3,1)--(6,1)
	(1,3)--(1,4)
	(2,2)--(2,4)
	(3,1)--(3,4)
	(4,0)--(4,4)
	(5,0)--(5,4)
	(6,1)--(6,4)
	(7,3)--(7,4);
\node[] at (.5, 3.5) {$1$};
\node[] at (1.5, 3.5) {$2$};
\node[] at (2.5, 3.5) {$3$};
\node[] at (3.5, 3.5) {$6$};
\node[] at (4.5, 3.5) {$7$};
\node[] at (5.5, 3.5) {$8$};
\node[] at (6.5, 3.5) {$12$};
\node[] at (1.5, 2.5) {$4$};
\node[] at (2.5, 2.5) {$5$};
\node[] at (3.5, 2.5) {$9$};
\node[] at (4.5, 2.5) {$11$};
\node[] at (5.5, 2.5) {$17$};
\node[] at (2.5, 1.5) {$10$};
\node[] at (3.5, 1.5) {$13$};
\node[] at (4.3, 1.7) {\tiny $14$};
\node[] at (4.7, 1.34) {\tiny $16$};
\node[] at (5.5, 1.5) {$18$};
\node[] at (3.5, .5) {$15$};
\node[] at (4.5, .5) {$19$};
\node[] at (-1.5,2) {$T=$};
\begin{scope}[shift={(13,0)}]
\draw (0,4)--(0,3)--(1,3)--(1,2)--(2,2)--(2,1)--(3,1)--(3,0)--(5,0)--(5,1)--(6,1)--(6,3)--(7,3)--(7,4)--(0,4)--cycle;
\draw (1,3)--(7,3)
	(2,2)--(6,2)
	(3,1)--(6,1)
	(1,3)--(1,4)
	(2,2)--(2,4)
	(3,1)--(3,4)
	(4,0)--(4,4)
	(5,0)--(5,4)
	(6,1)--(6,4)
	(7,3)--(7,4);
\node[] at (.5, 3.5) {$1$};
\node[] at (1.5, 3.5) {$2$};
\node[] at (2.5, 3.5) {$3$};
\node[] at (3.5, 3.5) {$6$};
\node[] at (4.5, 3.5) {$7$};
\node[] at (5.5, 3.5) {$8$};
\node[] at (6.5, 3.5) {$12$};
\node[] at (1.5, 2.5) {$4$};
\node[] at (2.5, 2.5) {$5$};
\node[] at (3.5, 2.5) {$9$};
\node[] at (4.5, 2.5) {$11$};
\node[] at (5.3, 2.7) {\tiny $16$};
\node[] at (5.7, 2.34) {\tiny $17$};
\node[] at (2.5, 1.5) {$10$};
\node[] at (3.5, 1.5) {$13$};
\node[] at (4.5, 1.5) {$14$};
\node[] at (5.5, 1.5) {$18$};
\node[] at (3.5, .5) {$15$};
\node[] at (4.5, .5) {$19$};
\node[] at (-1.5,2) {$T'=$};
\end{scope}
\end{tikzpicture}
$$

$$
\begin{tikzpicture}[scale=.45]
\draw (0,4)--(0,3)--(1,3)--(1,2)--(2,2)--(2,1)--(3,1)--(3,0)--(5,0)--(5,1)--(6,1)--(6,3)--(7,3)--(7,4)--(0,4)--cycle;
\draw (1,3)--(7,3)
	(2,2)--(6,2)
	(3,1)--(6,1)
	(1,3)--(1,4)
	(2,2)--(2,4)
	(3,1)--(3,4)
	(4,0)--(4,4)
	(5,0)--(5,4)
	(6,1)--(6,4)
	(7,3)--(7,4);
\node[] at (.5, 3.5) {$1$};
\node[] at (1.5, 3.5) {$2$};
\node[] at (2.5, 3.5) {$3$};
\node[] at (3.5, 3.5) {$6$};
\node[] at (4.5, 3.5) {$7$};
\node[] at (5.5, 3.5) {$8$};
\node[] at (6.5, 3.5) {$19$};
\node[] at (1.5, 2.5) {$4$};
\node[] at (2.5, 2.5) {$5$};
\node[] at (3.5, 2.5) {$9$};
\node[] at (4.5, 2.5) {$11$};
\node[] at (5.5, 2.5) {$14$};
\node[] at (2.5, 1.5) {$10$};
\node[] at (3.5, 1.5) {$12$};
\node[] at (4.5, 1.5) {$13$};
\node[] at (5.5, 1.5) {$15$};
\node[] at (3.5, .5) {$16$};
\node[] at (4.3, .7) {\tiny $17$};
\node[] at (4.7, .34) {\tiny $18$};
\node[] at (-1.5,2) {$T=$};
\begin{scope}[shift={(13,0)}]
\draw (0,4)--(0,3)--(1,3)--(1,2)--(2,2)--(2,1)--(3,1)--(3,0)--(5,0)--(5,1)--(6,1)--(6,3)--(7,3)--(7,4)--(0,4)--cycle;
\draw (1,3)--(7,3)
	(2,2)--(6,2)
	(3,1)--(6,1)
	(1,3)--(1,4)
	(2,2)--(2,4)
	(3,1)--(3,4)
	(4,0)--(4,4)
	(5,0)--(5,4)
	(6,1)--(6,4)
	(7,3)--(7,4)
	(5,0)--(6,0)--(6,1);
\node[] at (.5, 3.5) {$1$};
\node[] at (1.5, 3.5) {$2$};
\node[] at (2.5, 3.5) {$3$};
\node[] at (3.5, 3.5) {$6$};
\node[] at (4.5, 3.5) {$7$};
\node[] at (5.5, 3.5) {$8$};
\node[] at (6.5, 3.5) {$19$};
\node[] at (1.5, 2.5) {$4$};
\node[] at (2.5, 2.5) {$5$};
\node[] at (3.5, 2.5) {$9$};
\node[] at (4.5, 2.5) {$11$};
\node[] at (5.5, 2.5) {$14$};
\node[] at (2.5, 1.5) {$10$};
\node[] at (3.5, 1.5) {$12$};
\node[] at (4.5, 1.5) {$13$};
\node[] at (5.5, 1.5) {$15$};
\node[] at (3.5, .5) {$16$};
\node[] at (4.5, .5) {$17$};
\node[] at (5.5, .5) {$18$};
\node[] at (-1.5,2) {$T'=$};
\end{scope}
\end{tikzpicture}
$$
  \caption{Two examples of the map $T\mapsto T'$ in the proof of
    Lemma~\ref{lem:+k2}.}
  \label{fig:uncrowding1}
\end{figure}

It is easy to see that the map $T\mapsto T'$ is a bijection from
  $A$ to $B\cup C$, and the proof follows. 
\end{proof}

\begin{lem}\label{lem:+k1}
For $\lambda\in\Par^*_n$ and $1\le k\le n$,
  \[
   g^\lambda_k(+1)= \sum_{i=1}^k g^\lambda_{i,i}(+1).
  \]
\end{lem}
\begin{proof}
  The proof is similar to that of Lemma~\ref{lem:+k2}. 
It is enough to show that
  \begin{equation}
    \label{eq:2}
    g^\lambda_k(+1) = g^\lambda_{k,k}(+1)+ g^\lambda_{k-1}(+1),
  \end{equation}
  where $g^\lambda_0(+1)=0$. We prove \eqref{eq:2} by constructing a bijection
  from $A$ to $B$, where $A$ is the set of SBT
  of shape $\lambda$ with double
  cell in column $k$ but not in the diagonal cell $(k,k)$
  and $B$ is the set of SBT of shape $\lambda$ with double
  cell in column $k-1$.

  Suppose $T\in A$. Let $a<b$ be the integers in the double cell in $T$. Let $c$
  be the largest integer smaller than $b$ in the $(k-1)$th column of $T$. Define
  $T'$ to be the SBT obtained from $T$ by moving $a$ to the cell containing $c$.
  See Figure~\ref{fig:uncrowding2} for an example of this map.

\begin{figure}
$$
\begin{tikzpicture}[scale=.45]
\draw (0,4)--(0,3)--(1,3)--(1,2)--(2,2)--(2,1)--(3,1)--(3,0)--(5,0)--(5,1)--(6,1)--(6,3)--(7,3)--(7,4)--(0,4)--cycle;
\draw (1,3)--(7,3)
	(2,2)--(6,2)
	(3,1)--(6,1)
	(1,3)--(1,4)
	(2,2)--(2,4)
	(3,1)--(3,4)
	(4,0)--(4,4)
	(5,0)--(5,4)
	(6,1)--(6,4)
	(7,3)--(7,4);
\node[] at (.5, 3.5) {$1$};
\node[] at (1.5, 3.5) {$2$};
\node[] at (2.5, 3.5) {$3$};
\node[] at (3.5, 3.5) {$5$};
\node[] at (4.5, 3.5) {$8$};
\node[] at (5.5, 3.5) {$10$};
\node[] at (6.5, 3.5) {$11$};
\node[] at (1.5, 2.5) {$4$};
\node[] at (2.5, 2.5) {$6$};
\node[] at (3.3, 2.7) {\tiny $9$};
\node[] at (3.7, 2.34) {\tiny $12$};
\node[] at (4.5, 2.5) {$15$};
\node[] at (5.5, 2.5) {$16$};
\node[] at (2.5, 1.5) {$7$};
\node[] at (3.5, 1.5) {$13$};
\node[] at (4.5, 1.5) {$17$};
\node[] at (5.5, 1.5) {$19$};
\node[] at (3.5, .5) {$14$};
\node[] at (4.5, .5) {$18$};
\node[] at (-1.5,2) {$T=$};
\end{tikzpicture}
\qquad\quad
\begin{tikzpicture}[scale=.45]
\draw (0,4)--(0,3)--(1,3)--(1,2)--(2,2)--(2,1)--(3,1)--(3,0)--(5,0)--(5,1)--(6,1)--(6,3)--(7,3)--(7,4)--(0,4)--cycle;
\draw (1,3)--(7,3)
	(2,2)--(6,2)
	(3,1)--(6,1)
	(1,3)--(1,4)
	(2,2)--(2,4)
	(3,1)--(3,4)
	(4,0)--(4,4)
	(5,0)--(5,4)
	(6,1)--(6,4)
	(7,3)--(7,4);
\node[] at (.5, 3.5) {$1$};
\node[] at (1.5, 3.5) {$2$};
\node[] at (2.5, 3.5) {$3$};
\node[] at (3.5, 3.5) {$5$};
\node[] at (4.5, 3.5) {$8$};
\node[] at (5.5, 3.5) {$10$};
\node[] at (6.5, 3.5) {$11$};
\node[] at (1.5, 2.5) {$4$};
\node[] at (2.5, 2.5) {$6$};
\node[] at (3.5, 2.5) {$12$};
\node[] at (4.5, 2.5) {$15$};
\node[] at (5.5, 2.5) {$16$};
\node[] at (2.3, 1.7) {\small $7$};
\node[] at (2.7, 1.34) {\small $9$};
\node[] at (3.5, 1.5) {$13$};
\node[] at (4.5, 1.5) {$17$};
\node[] at (5.5, 1.5) {$19$};
\node[] at (3.5, .5) {$14$};
\node[] at (4.5, .5) {$18$};
\node[] at (-1.5,2) {$T'=$};
\end{tikzpicture}$$
  \caption{An example of the map $T\mapsto T'$ in the proof of
    Lemma~\ref{lem:+k1}.}
  \label{fig:uncrowding2}
\end{figure}
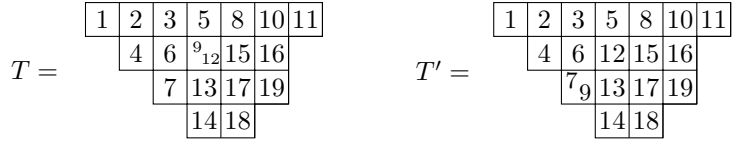

  It is easy to see that the map $T\mapsto T'$ is a bijection from $A$ to $B$,
  and the proof follows.
\end{proof}

\begin{lem}\label{lem:i,i+1}
  For $\lambda\in\Par^*_n$ and $1\le i \le n-1$,
\[
g^{\lambda\cup\{(i+1,i)\}} = g^\lambda_{i,i}(+1) + g^\lambda_{i+1,i+1}(+1) .
\]
\end{lem}
\begin{proof}
  It suffices to find a bijection from $A$ to $B\cup C$, where $A$ is the set of
  SYTs of shape $\lambda\cup\{(i+1,i)\}$, $B$ is the set of SBTs of shape
  $\lambda$ with double cell $(i,i)$, and $C$ is the set of SBTs of shape
  $\lambda$ with double cell $(i+1,i+1)$.

  Let $T\in A$. Let $a$ and $b$ be the integers in cell $(i+1,i)$ and $(i,i+1)$
  of $T$, respectively. Define $T'$ to be the SBT obtained from $T$ by removing
  cell $(i+1,i)$ and putting $a$ into cell $(i,i)$ if $a<b$ and into cell
  $(i+1,i+1)$ if $a>b$. See Figure~\ref{fig:uncrowding3} for examples of this
  map.

\begin{figure}
\centering
\begin{tikzpicture}[scale=.45]
\draw (0,3)--(0,2)--(1,2)--(1,1)--(2,1)--(2,0)--(3,0)--(3,1)--(4,1)--(4,3)--(0,3)--cycle;
\draw (1,2)--(4,2)
	(2,1)--(3,1)
	(1,2)--(1,3)
	(2,1)--(2,3)
	(3,1)--(3,3)
	(1,1)--(1,0)--(2,0);
\node[] at (.5, 2.5) {$1$};
\node[] at (1.5, 2.5) {$2$};
\node[] at (2.5, 2.5) {$4$};
\node[] at (3.5, 2.5) {$7$};
\node[] at (1.5, 1.5) {$3$};
\node[] at (2.5, 1.5) {$\bf 6$};
\node[] at (3.5, 1.5) {$9$};
\node[] at (2.5, .5) {$8$};
\node[] at (1.5, .5) {$\bf 5$};
\node[] at (5, 1.5) {$\longleftrightarrow$};
\begin{scope}[shift={(6,0)}]
\draw (0,3)--(0,2)--(1,2)--(1,1)--(2,1)--(2,0)--(3,0)--(3,1)--(4,1)--(4,3)--(0,3)--cycle;
\draw (1,2)--(4,2)
	(2,1)--(3,1)
	(1,2)--(1,3)
	(2,1)--(2,3)
	(3,1)--(3,3);
\node[] at (.5, 2.5) {$1$};
\node[] at (1.5, 2.5) {$2$};
\node[] at (2.5, 2.5) {$4$};
\node[] at (3.5, 2.5) {$7$};
\node[] at (1.3, 1.65) {$3$};
\node[] at (1.75, 1.35) {$\bf 5$};
\node[] at (2.5, 1.5) {$6$};
\node[] at (3.5, 1.5) {$9$};
\node[] at (2.5, .5) {$8$};
\end{scope}
\begin{scope}[shift={(13,0)}]
\draw (0,3)--(0,2)--(1,2)--(1,1)--(2,1)--(2,0)--(3,0)--(3,1)--(4,1)--(4,3)--(0,3)--cycle;
\draw (1,2)--(4,2)
	(2,1)--(3,1)
	(1,2)--(1,3)
	(2,1)--(2,3)
	(3,1)--(3,3)
	(1,1)--(1,0)--(2,0);
\node[] at (.5, 2.5) {$1$};
\node[] at (1.5, 2.5) {$2$};
\node[] at (2.5, 2.5) {$4$};
\node[] at (3.5, 2.5) {$7$};
\node[] at (1.5, 1.5) {$3$};
\node[] at (2.5, 1.5) {$\bf 5$};
\node[] at (3.5, 1.5) {$9$};
\node[] at (1.5, .5) {$\bf 6$};
\node[] at (2.5, .5) {$8$};
\node[] at (5, 1.5) {$\longleftrightarrow$};
\end{scope}
\begin{scope}[shift={(19,0)}]
\draw (0,3)--(0,2)--(1,2)--(1,1)--(2,1)--(2,0)--(3,0)--(3,1)--(4,1)--(4,3)--(0,3)--cycle;
\draw (1,2)--(4,2)
	(2,1)--(3,1)
	(1,2)--(1,3)
	(2,1)--(2,3)
	(3,1)--(3,3);
\node[] at (.5, 2.5) {$1$};
\node[] at (1.5, 2.5) {$2$};
\node[] at (2.5, 2.5) {$4$};
\node[] at (3.5, 2.5) {$7$};
\node[] at (1.5, 1.5) {$3$};
\node[] at (2.5, 1.5) {$5$};
\node[] at (3.5, 1.5) {$9$};
\node[] at (2.75, .35) {$8$};
\node[] at (2.3, .65) {$\bf 6$};
\end{scope}
\end{tikzpicture}
\caption{Two examples of the map $T\mapsto T'$ in the proof of
    Lemma~\ref{lem:i,i+1}.}
  \label{fig:uncrowding3}
\end{figure}
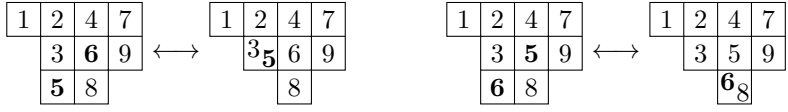

It is easy to see that the map $T\mapsto T'$ is a
  desired bijection.
\end{proof}

\begin{lem}\label{lem:SSYT_sum1}
  For $\lambda\in\Par^*_n$,
  \[
  \sum_{i=1}^n g^\lambda_i(+1)=
  \frac{1}{2} \sum_{(i,j)\in \NE(\lambda)} g^{\lambda\cup\{(i,j)\}}
  +\frac{1}{2} \sum_{i=0}^{n-1} (n-i) g^{\lambda\cup\{(i+1,i)\}}.
  \]
\end{lem}
\begin{proof}
  Let $L$ be the left hand side of the identity. Then by Lemma~\ref{lem:+k1},
\begin{align*}
  2L - g^\lambda_n(+1) &= \sum_{i=1}^n (2n-2i+1)g^\lambda_{i,i}(+1)\\
  &= \sum_{i=1}^n (n-i)g^\lambda_{i,i}(+1)
    + \sum_{i=0}^{n-1} (n-i)g^\lambda_{i+1,i+1}(+1)\\
  &= ng^\lambda_{1,1}(+1)+
    \sum_{i=1}^{n-1} (n-i)(g^\lambda_{i,i}(+1)+g^\lambda_{i+1,i+1}(+1)) \\
  &= ng^\lambda_{1,1}(+1)+\sum_{i=1}^{n-1} (n-i)g^{\lambda\cup \{(i+1,i)\}},
\end{align*}
where the last equality follows from Lemma~\ref{lem:i,i+1}.
Since
$g^\lambda_{1,1}(+1)= g^{\lambda\cup\{(1,0)\}}$
and, by Lemma~\ref{lem:+k2},
\[
 g^\lambda_n(+1) = \sum_{(i,j)\in \NE(\lambda)} g^{\lambda\cup\{(i,j)\}},
\]
hence we obtain the desired formula for $L$.
\end{proof}

\begin{lem}\label{lem:SSYT_sum2}
  For $\lambda\in\Par^*_n$,
\[
  \sum_{i=0}^{n-1} (n-i)g^{\lambda\cup\{(i+1,i)\}}=
    n (|\lambda|+1)g^\lambda
  - \sum_{(i,j)\in \NE(\lambda)} \lambda_i \, g^{\lambda\cup\{(i,j)\}}.
\]
  \end{lem}
\begin{proof}
  The left hand side can be rewritten as 
  \[
\sum_{i=1}^{n} i\cdot g^{\lambda\cup\{(n-i+1,n-i)\}}.
  \]
  By Proposition~\ref{prop:SSYT_int2}, this is equal to
  \begin{align}
    \notag
&\frac{(|\lambda|+1)!}{\prod_{k=1}^n (\lambda_k-1)!} \sum_{i=1}^{n}
\int_{0\le x_1\le \cdots \le x_n\le 1}(ix_{i+1}-ix_i)\overline{a}_{\lambda-(1^n)}(x)d x\\
    \notag
    &=\frac{(|\lambda|+1)!}{\prod_{k=1}^n (\lambda_k-1)!}
      \int_{0\le x_1\le \cdots \le x_n\le 1}
      (n-x_1-\dots-x_n)\overline{a}_{\lambda-(1^n)}(x)d x\\
\label{eq:e1}
&=n(|\lambda|+1)g^\lambda-
\frac{(|\lambda|+1)!}{\prod_{k=1}^n (\lambda_k-1)!}\int_{0\le x_1\le \cdots \le x_n\le
  1}e_1(x)\overline{a}_{\lambda-(1^n)}(x)d x,
  \end{align}
  where the last equality follows from Proposition~\ref{prop:SSYT_int} and
  $e_1(x)=x_1+\dots+x_n$. Let $\lambda=\delta_{n+1}+\mu$, or equivalently,
  $\lambda-(1^n)=\delta_{n}+\mu$, where we consider $\mu$ as a partition shape.
  Then the integrand of the $q$-integral in \eqref{eq:e1} is
\[
e_1(x)\overline{a}_{\lambda-(1^n)}(x)= e_1(x) \overline{a}_{\mu+\delta_n}(x)
= e_1(x) s_\mu(x) \overline{a}_{\delta_n}(x),
\]
where $s_\mu(x)=a_{\mu+\delta_n}(x)/a_{\delta_n}(x)$ is the Schur function.

By the Pieri rule,
\[
e_1(x) s_\mu(x) = \sum_{(i,j)\in \NE(\mu)} s_{\mu\cup\{(i,j)\}}(x).
\]
Notice that since $\mu\in\Par_n$, there is no $(n+1,j)\in \NE(\mu)$.
Therefore,
\[
  e_1(x) s_\mu(x) \overline{a}_{\delta_n}(x)=
\sum_{(i,j)\in \NE(\mu),\, i\le n} s_{\mu\cup\{(i,j)\}}(x)\overline{a}_{\delta_n}(x) =
\sum_{(i,j)\in \NE(\mu),\, i\le n} \overline{a}_{\mu\cup\{(i,j)\}+\delta_n}(x),
\]
and the second term of \eqref{eq:e1} (without the minus sign) is equal to
\[
\sum_{(i,j)\in \NE(\mu),\, i\le n}\frac{(|\lambda|+1)!}{\prod_{k=1}^n (\lambda_k-1)!}
\int_{0\le x_1\le \cdots \le x_n\le 1}\overline{a}_{\delta_{n+1}+(\mu\cup\{(i,j)\})-(1^n)}(x)d x.
\]
By Proposition~\ref{prop:SSYT_int}, this is equal to
\[
\sum_{(i,j)\in \NE(\mu),\, i\le n}\lambda_i g^{\delta_{n+1}+(\mu\cup\{(i,j)\})}.
\]
Recall that $\lambda$ is a shifted shape and $\mu$ is a partition shape. Since
$(i,j)\in \NE(\mu)$ and $i\le n$ if and only if $(i,n+j)\in \NE(\lambda)$, and
\[
\delta_{n+1}+(\mu\cup\{(i,j)\}) = \lambda\cup\{(i,n+j)\},
\]
the proof is complete.
\end{proof}

The following proposition allows us to express the number of SBTs of any shifted
shape as a sum of the numbers of SYTs of shifted shapes.

\begin{prop}\label{prop:g+1}
  For $\lambda\in\Par^*_n$,
\[
g^\lambda(+1)
= \frac{n(|\lambda|+1)}{2}g^\lambda
+\frac{1}{2}\sum_{(i,j)\in \NE(\lambda)}
  (2j-2n-1-\lambda_i) g^{\lambda\cup\{(i,j)\}}.
\]
\end{prop}
\begin{proof}
  Considering the column containing the extra cell, we have
  \[
    g^\lambda(+1)=\sum_{i=1}^{n}g^\lambda_i(+1)
    + \sum_{i=n+1}^{\lambda_1}g^\lambda_i(+1).
  \]
By Lemmas~\ref{lem:SSYT_sum1} and \ref{lem:SSYT_sum2},
\[
\sum_{i=1}^{n}g^\lambda_i(+1)
 = \frac{n(|\lambda|+1)}{2}g^\lambda
+\frac{1}{2}\sum_{(i,j)\in \NE(\lambda)}
  (1-\lambda_i) g^{\lambda\cup\{(i,j)\}}.
\]
By Lemma~\ref{lem:+k2},
  \[
    \sum_{i=n+1}^{\lambda_1}g^\lambda_i=
    \sum_{(i,j)\in \NE(\lambda)} (j-n-1)g^{\lambda\cup\{(i,j)\}}.
  \]
The proof follows by adding the above two equations.
\end{proof}

Now we can prove Theorem~\ref{thm:main} easily. 

\begin{proof}[Proof of Theorem~\ref{thm:main}]
Proposition~\ref{prop:g+1} can be restated as
\[
g^\lambda(+1)
= \frac{n(|\lambda|+1)}{2}g^\lambda
+\frac{1}{2} \sum_{k:\lambda_k\le \lambda_{k-1}-2}(\lambda_k-2n+2k-1)
 g^{\lambda\cup\{(k,k+\lambda_k)\}}.
\]
Using the formula \eqref{eqn:Thrall}, we obtain the result.
\end{proof}

%----------------------------------------------------------------------------------------------------------

\section{Enumeration of special classes of shifted SBTs}
\label{sec:enum-stand-barely}
%----------------------------------------------------------------------------------------------------------

In this section we compute the expectation $\EE(Y)$ for the lower intervals
$[\emptyset,\lambda]$ described in Theorems~\ref{thm:RTY1},
\ref{thm:shifted-balanced}, and \ref{thm:trapezoidal}: the cases that $\lambda$
is $\delta_d +\delta_e (a^a)$, $\lambda$ is a balanced shifted shape, and
$\lambda$ is the trapezoidal shape $(m+2n, m+2n-2,\dots, m+2)$.

Recall that since
\[
\EE(Y) = \frac{g^\lambda(+1)}{(|\lambda|+1) g^\lambda },
\]
and $g^\lambda$ is known, computing $\EE(Y)$ is equivalent to computing
$g^\lambda(+1)$.

The following theorem computes the expectation $\EE(Y)$ for the lower interval
$[\emptyset,\lambda]$ when $\lambda$ is $\delta_d +\delta_e (a^a)$. This shows
the formula for $\EE(Y)$ in Theorem~\ref{thm:RTY1}.

\begin{thm}\label{thm:conj24}
Let $\lambda=\delta_d +\delta_e  (a^a)$ with $a,d,e\ge 1$ and $d>a(e-1)+1$. Then
$$ g^\lambda (+1)=(|\lambda|+1)\frac{d+a(e-1)}{4} g^\lambda. $$
Equivalently, the expectation $\EE(Y)$ for the lower interval
$[\emptyset,\lambda]$ is equal to 
\[
\EE(Y) = \frac{d+a(e-1)}{4 }.
\]
\end{thm}

\begin{proof}
We want to show that
$$
\frac{g^\lambda (+1)}{(|\lambda|+1)g^\lambda}=\frac{d+a(e-1)}{4}.
$$
By Proposition~\ref{prop:g+1} we have (using $n=d-1$)
\begin{align*}
\frac{g^\lambda (+1)}{(|\lambda|+1)g^\lambda}&=\frac{d-1}2+\frac 12
(2-d+a(e-1))\,G,\\
\intertext{where}
G&=\sum_{(i,j)\in\NE(\lambda)}
\frac{g^{\lambda\cup \{(i,i+\lambda_i) \} }}{(|\lambda|+1)g^\lambda}.
\end{align*}
(Notice that the factor $2-d+a(e-1)$ was pulled out of the sum.
The shifted-balanced shape of the partition $\lambda$ made this possible!)
Since
\begin{equation*}
\frac{d+a(e-1)}{4}-\frac{d-1}2=\frac{2-d+a(e-1)}4,
\end{equation*}
we thus need to show that $G=1/2$.
Now, to compute $G$, we observe that 
for $\lambda=\delta_d +\delta_e ( a^a)$ we have
$\NE(\lambda) =\{ (i, i+\lambda_i)~|~ i=ja +1 \text{ for } 0\le j\le e-1\}$,
where $\lambda_i = d-1 +a(e-1-2j)$.
By the shifted hook-length formula \eqref{eqn:Thrall}, and a rewriting
of the product in terms of Pochhammer symbols (cf.\ \cite{AAR}) we have
\begin{equation*}
\frac{g^{\lambda\cup \{(i,i+\lambda_i) \} }}{(|\lambda|+1)g^\lambda}
=\frac{(2j)!}{2^{4e-3}(j!)^2} \frac{(2e-2-2j)!}{((e-j-1)!)^2} \frac{\left(\frac{2d-1}{a}+1-2j \right)_{2e-2}}{\left(\left(\frac{2d-1}{2a}+1-j \right)_{e-1} \right)^2}
\frac{(2d-1+2a(e-1-2j))}{(2d-1-2ja)},
\end{equation*}
for each $i=ja +1$, $0\le j\le e-1$.
The identity that we need to show to establish the theorem is
\begin{equation}\label{eqn:identitytoshow1}
2^{4e-4}=\sum_{j=0}^{e-1} \frac{(2j)!}{(j!)^2} \frac{(2e-2-2j)!}{((e-j-1)!)^2} \frac{\left(\frac{2d-1}{a}+1-2j \right)_{2e-2}}{\left(\left(\frac{2d-1}{2a}+1-j \right)_{e-1} \right)^2}
\frac{(2d-1+2a(e-1-2j))}{(2d-1-2ja)}.
\end{equation}
The right hand side of \eqref{eqn:identitytoshow1} can be rewritten as
\begin{align}\label{eqn:identitytoshow2}
&\frac{(2e-2)!}{2((e-1)!)^2}\frac{\left( \frac{2d-1}{a}+1\right)_{2e-2}}{\left(\left(\frac{2d-1}{2a}+1 \right)_{e-1} \right)^2}\frac{(1-2d+2a(1-e))}{(1-2d)}\\\notag
&\times \sum_{j=0}^{e-1}\frac{\left(\frac{1-2d}{2a}+1-e, \frac{1-2d}{4a}+\frac{3}{2}-\frac{e}{2}, \frac{1}{2}, 1-e, \frac{1-2d}{2a}+\frac{1}{2} \right)_j}{\left(1, \frac{1-2d}{4a}+\frac{1}{2}-\frac{e}{2}, \frac{3}{2}-e+\frac{1-2d}{2a}, \frac{1-2d}{2a}+1, \frac{3}{2}-e \right)_j}.
\end{align}
To evaluate the terminating series in \eqref{eqn:identitytoshow2}, we apply Dougall's summation formula \cite[(2.1.7)]{GR}
\begin{align*}
&{}_5 F_4\!\left[\begin{matrix} \mathsf{a}, 1+\frac{1}{2}\mathsf{a}, \mathsf{b},\mathsf{c},\mathsf{d}\\
\frac{1}{2}\mathsf{a}, 1+\mathsf{a}-\mathsf{b}, 1+\mathsf{a}-\mathsf{c}, 1+\mathsf{a}-\mathsf{d} \end{matrix}\,;1\right]\\
&=\frac{\Gamma(1+\mathsf{a}-\mathsf{b})\Gamma(1+\mathsf{a}-\mathsf{c})\Gamma(1+\mathsf{a}-\mathsf{d})\Gamma(1+\mathsf{a}-\mathsf{b}-\mathsf{c}-\mathsf{d})}{\Gamma(1+\mathsf{a})\Gamma(1+\mathsf{a}-\mathsf{b}-\mathsf{c})\Gamma(1+\mathsf{a}-\mathsf{b}-\mathsf{d})\Gamma(1+\mathsf{a}-\mathsf{c}-\mathsf{d})},
\end{align*}
with $(\mathsf{a},\mathsf{b},\mathsf{c},\mathsf{d})\mapsto \left( \frac{1-2d}{2a}+1-e, \frac{1}{2}, 1-e, \frac{1-2d}{2a}+\frac{1}{2}\right)$. As a result, we obtain
$$\sum_{j=0}^{e-1}\frac{\left(\frac{1-2d}{2a}+1-e, \frac{1-2d}{4a}+\frac{3}{2}-\frac{e}{2}, \frac{1}{2}, 1-e, \frac{1-2d}{2a}+\frac{1}{2} \right)_j}{\left(1, \frac{1-2d}{4a}+\frac{1}{2}-\frac{e}{2}, \frac{3}{2}-e+\frac{1-2d}{2a}, \frac{1-2d}{2a}+1, \frac{3}{2}-e \right)_j}
=\frac{\left( \frac{1-2d}{2a}+2-e\right)_{e-1}(1-e)_{e-1}}{\left( \frac{1-2d}{2a}+\frac{3}{2}-e\right)_{e-1}\left(\frac{3}{2}-e \right)_{e-1}}.$$
Now it is not very hard to show that
$$\frac{(2e-2)!}{2((e-1)!)^2}\frac{\left( \frac{2d-1}{a}+1\right)_{2e-2}}{\left(\left(\frac{2d-1}{2a}+1 \right)_{e-1} \right)^2}\frac{(1-2d+2a(1-e))}{(1-2d)} \frac{\left( \frac{1-2d}{2a}+2-e\right)_{e-1}(1-e)_{e-1}}{\left( \frac{1-2d}{2a}+\frac{3}{2}-e\right)_{e-1}\left(\frac{3}{2}-e \right)_{e-1}}=2^{4e-4},$$
which completes the proof.
\end{proof}

\begin{figure}[ht]
\begin{tikzpicture}[scale=.43]
\draw[fill=gray!35] (9,2) rectangle (11,9);
\draw[fill=gray!35] (11,4) rectangle (12,9);
\draw[fill=gray!35] (12,5) rectangle (15,9);
\draw[fill=gray!35] (15,8) rectangle (16,9);
\draw (0,8)--(0,9)--(16,9)--(16,8)--(15,8)--(15,5)--(12,5)--(12,4)--(11,4)--(11,2)--(9,2)--(9,0)--(8,0)--(8,1)--(7,1)--(7,2)--(6,2)--(6,3)--(5,3)--(5,4)--(4,4)--(4,5)--(3,5)--(3,6)--(2,6)--(2,7)--(1,7)--(1,8)--(0,8)--cycle;
\draw (1,8)--(15,8);
\draw (2,7)--(15,7);
\draw (3,6)--(15,6);
\draw (4,5)--(12,5);
\draw (5,4)--(11,4);
\draw (6,3)--(11,3);
\draw (7,2)--(9,2);
\draw (8,1)--(9,1);
\draw (1,8)--(1,9);
\draw (2,7)--(2,9);
\draw (3,5)--(3,9);
\draw (4,4)--(4,9);
\draw (5,3)--(5,9);
\draw (6,2)--(6,9);
\draw (7,1)--(7,9);
\draw (8,0)--(8,9);
\draw (10,2)--(10,9);
\draw (9,2)--(9,9);
\draw (11,4)--(11,9);
\draw (12,5)--(12,9);
\draw (13,5)--(13,9);
\draw (14,5)--(14,9);
\draw (15,8)--(15,9);
\draw[thick, dotted] (9,2)--(16,9);
\end{tikzpicture}
\caption{$\lambda=\delta_{10} +\nu$ with $k=7$}
\end{figure}
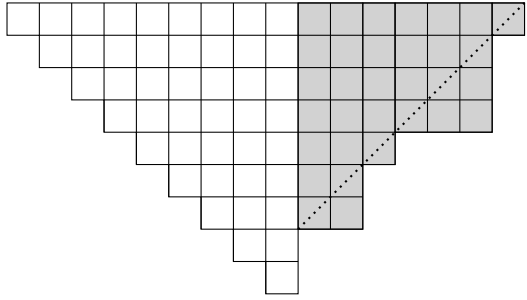

The following theorem computes the expectation $\EE(Y)$ for the lower interval
$[\emptyset,\lambda]$ when $\lambda$ is a balanced shifted shape. This shows the
formula for $\EE(Y)$ in Theorem~\ref{thm:shifted-balanced}.

\begin{thm}\label{thm:thm4.2}
  Let $0\le k<n$ and let $\nu$ be a regular partition which as a straight shape
  is balanced with height and width both equal to $k$.
\begin{enumerate}
\item Let $\lambda=\delta_{n+1}+\nu$. Then
$$g^\lambda(+1)=(|\lambda|+1)\frac{n+1+k}{4} g^\lambda. $$
Equivalently, the expectation $\EE(Y)$ for the lower interval
$[\emptyset,\lambda]$ is equal to 
\[
\EE(Y) = \frac{n+1+k}{4}.
\]
\item Let $\lambda=\delta_{n+1}+(n-1-k)^n + \nu$. Then
$$g^\lambda(+1)=(|\lambda|+1)\frac{n}{2} g^\lambda. $$
Equivalently, the expectation $\EE(Y)$ for the lower interval
$[\emptyset,\lambda]$ is equal to 
\[
\EE(Y) = \frac{n}{2}.
\]
\end{enumerate}
\end{thm}

\begin{proof}
Since $\nu$ is balanced, we can write
$\nu=((a_1+a_2+\cdots +a_\ell)^{a_1}, (a_2+\cdots +a_\ell)^{a_2},\dots, (a_\ell)^{a_\ell})$, for $a_i$'s such that $a_i\ge 1$ for all $1\le i\le \ell$ and $a_1+\cdots +a_\ell=k$.
Then $\NE(\lambda)=\{ (r_i , c_i)~|~ r_i =a_1+\cdots +a_i +1,~c_i = r_i+\lambda_{r_i}, \text{ for } i=0,\dots, \ell\}.$

In the case when  $\lambda=\delta_{n+1}+\nu$,  using the fact that $\lambda_{a_1+\cdots +a_i +1}=n+k-2(a_1+\cdots +a_i)$, for $i=0,\dots, \ell$,  by expressing $g^{\lambda\cup\{(r_i ,c_i ) \}}$ in terms of $g^\lambda$ and comparing the terms in  equation \eqref{eq:g+1},
the identity that we need to prove becomes
\begin{multline}\label{eqn:thm4.2}
1= \sum_{i=0}^\ell \frac{\frac{1}{2}a_i}{a_i} \frac{(a_i +\frac{1}{2}a_{i-1})}{(a_i +a_{i-1})}\frac{(a_i +a_{i-1}+\frac{1}{2}a_{i-2})}{(a_i +a_{i-1}+a_{i-2})}\cdots
\frac{(a_i +a_{i-1}+\cdots + a_2 +\frac{1}{2}a_{1})}{(a_i +a_{i-1}+\cdots +a_2+a_1)}\\
\times \frac{\frac{1}{2}a_{i+1}}{a_{i+1}} \frac{(a_{i+1} +\frac{1}{2}a_{i+2})}{(a_{i+1} +a_{i+2})}\frac{(a_{i+1} +a_{i+2}+\frac{1}{2}a_{i+3})}{(a_{i+1} +a_{i+2}+a_{i+3})}\cdots
\frac{(a_{i+1} +a_{i+2}+\cdots + a_{\ell-1} +\frac{1}{2}a_{\ell})}{(a_{i+1} +a_{i+2}+\cdots +a_{\ell -1}+a_\ell)}\\
\times \frac{(n+\frac{1}{2}-r_i+a_\ell +a_{\ell -1 }+\cdots +a_{i+1})}{(n+\frac{1}{2}-r_i)}\\
\times \frac{(n+\frac{1}{2}-r_i+\frac{1}{2}a_\ell)}{(n+\frac{1}{2}-r_i+a_\ell)}\frac{(n+\frac{1}{2}-r_i+a_\ell +\frac{1}{2}a_{\ell-1})}{(n+\frac{1}{2}-r_i+a_\ell+a_{\ell-1})}\cdots
\frac{(n+\frac{1}{2}-r_i+a_\ell +a_{\ell -1}+\cdots +a_2 +\frac{1}{2}a_{1})}{(n+\frac{1}{2}-r_i+a_\ell+a_{\ell-1}+\cdots +a_2 +a_1)},
\end{multline}
where $r_i =a_1+\cdots +a_i +1$.
This identity can be proved by partial fraction decomposition, or by Lagrange interpolation.
The latter says that if $f(x)$ is a polynomial of degree $\ell$, then
\begin{equation*}
f(x)=\sum_{i=0}^\ell f(\mathsf{b}_i)\prod_{0\le j\le \ell,\, j\neq i}\frac{x-\mathsf{b}_j}{\mathsf{b}_i-\mathsf{b}_j},
\end{equation*}
where the $\mathsf{b}_i$'s are all distinct. Dividing both sides of this identity by $f(x)$,
choosing $f(x)=\prod_{j=0}^{\ell-1}(x-\mathsf{c}_j)$ and letting $x\to\infty$ we obtain
the well-known formula
\begin{equation*}
1=\sum_{i=0}^\ell \frac{\prod_{0\le j\le \ell-1}(\mathsf{b}_i-\mathsf{c}_j)}
{\prod_{0\le j\le\ell,\, j\neq i}(\mathsf{b}_i-\mathsf{b}_j)}.
\end{equation*}
By making the substitutions $\mathsf{b}_i\mapsto (b_i-u/2)^2$ and $\mathsf{c}_j\mapsto(c_j-u/2)^2$,
for $0\le i\le \ell$ and $0\le j\le \ell-1$,
we  obtain the equivalent identity
\begin{equation}\label{eq:li}
1=\sum_{i=0}^\ell \frac{\prod_{0\le j\le \ell-1}(b_i-c_j)(u-b_i-c_j)}
{\prod_{0\le j\le \ell,\, j\neq i}(b_i-b_j)(u-b_i-b_j)}.
\end{equation}
Now, the sum in \eqref{eqn:thm4.2} can be evaluated by specializing the parameters in \eqref{eq:li}
as $b_i\mapsto a_1+\cdots+a_i$, for $0\le i\le \ell$,
$c_j\mapsto (a_1+\cdots+a_j)+a_{j+1}/2$, for $0\le j\le \ell-1$, and
$u\mapsto a_1+\dots+a_l+n-1/2$.
This establishes the first part (1) of the theorem.

For the second part (2) of the theorem, we instead 
set $\lambda=\delta_{n+1}+(n-1-k)^n + \nu$, then $\lambda_{r_i}= 2n+1-2 r_i$  where $r_i = a_1+\cdots + a_i +1$, for  $i=0,\dots, \ell$.
By using this fact in Proposition \ref{prop:g+1}, we obtain
$$g^\lambda(+1)=(|\lambda|+1)\frac{n}{2} g^\lambda$$
which completes the proof.
\end{proof}

Finally we consider the case that $\lambda$ is the trapezoidal shape $(m+2n,
m+2n-2,\dots, m+2)$ whose typical diagram is as shown in
Figure~\ref{fig:trapezoidal}.

\begin{figure}[ht]
\begin{tikzpicture}[scale=.43]
\draw (3,3)--(2,3)--(2,4)--(1,4)--(1,5)--(0,5)--(0,6)--(19,6)--(19,5)--(18,5)--(18,4)--(17,4)--(17,3)--(16,3)--(16,2)--(15,2)--(15,1)--(14,1)--(14,0)--(5,0)--(5,1)--(4,1)--(4,2)--(3,2)--(3,3)--cycle;
%\draw (4,2)--(4,1)--(5,1)--(5,0)--(14,0)--(14,1)--(15,1)--(15,2);
\draw [snake=brace,color=blue] (0, 6.2)--(5.9, 6.2);
\draw [snake=brace,color=blue] (6.1, 6.2)--(12.9, 6.2);
\draw [snake=brace,color=blue] (13.1, 6.2)--(19, 6.2);
\draw[dashed] (6,0)--(6,6);
\draw[dashed] (13,0)--(13,6);
\node[] at (3, 6.85) {$n$};
\node[] at (9.5, 6.85) {$m$};
\node[] at (16, 6.85) {$n$};
\end{tikzpicture}
\caption{$\lambda=(m+2n, m+2n-2,\dots, m+2)$}
\label{fig:trapezoidal}
\end{figure}
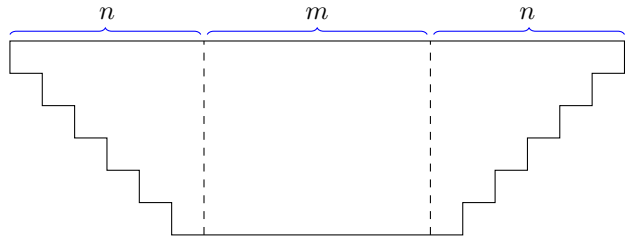

Note that in this case, by \eqref{eqn:Thrall},
\begin{equation}\label{eqn:syt}
g^\lambda= \frac{(nm+n^2+n)!}{\prod_{i=1}^n (m+2i)!} \prod_{i=1}^{n-1}\dfrac{i !  (m+1+i)!}{(m+1+2 i )!}.
\end{equation}

The following theorem shows the formula for $\EE(Y)$ in
Theorem~\ref{thm:trapezoidal}.

\begin{thm}\label{thm:g+1 trapezoidal}
  Let $\lambda=(m+2n, m+2n-2,\dots, m+2)$. Then
$$g^\lambda(+1)= (|\lambda|+1)\frac{|\lambda|}{\lambda_1 +1 } g^\lambda= \frac{(mn+n^2+n)(mn+n^2+n+1)!}{(m+2n+1)\prod_{i=1}^n (m+2i)!}\prod_{i=1}^{n-1}\frac{i! (m+1+i)!}{(m+1+2i)!}.$$
Equivalently, the expectation $\EE(Y)$ for the lower interval
$[\emptyset,\lambda]$ is equal to 
\[
\EE(Y) = \frac{g^\lambda(+1)}{(|\lambda|+1) g^\lambda } =
\frac{|\lambda|}{\lambda_1 +1 }.
\]

\end{thm}
\begin{proof}
Note that $\NE(\lambda)=\{ (i+1,  \lambda_{i+1} +1) ~|~ 0\le i\le n-1\}$, where $\lambda_{i+1}=m+2n-2i$, for $0\le i\le n-1$.
By applying 
\begin{multline}\label{eqn:lambda_i}
\frac{g^{\lambda\cup \{(i+1, \lambda_{i+1} +1)\}}}{g^\lambda (|\lambda|+1)}\\
=  2  \frac{(2i)!}{(i!)^2} \frac{(2n-2i-1)!}{((n-i-1)!)^2}\frac{((m+2n-i)!)^2}{((m+n-i)!)^2}
\frac{(2m+2n-2i+1)!}{(2m+4n-2i+1)!}\frac{(2m+4n-4i+1)}{(m+2n-2i)(m+2n-2i+1)},
\end{multline}
for $0\le i\le n-1$, in equation \eqref{eq:g+1}, 
 the identity that we need to prove is
\begin{multline}\label{eqn:mainidentity}
\frac{n}{2(m+2n+1)}=\\
\sum_{i=0}^{n-1}\frac{(2i)!}{(i!)^2} \frac{(2n-2i-1)!}{((n-i-1)!)^2}\frac{((m+2n-i)!)^2}{((m+n-i)!)^2}
\frac{(2m+2n-2i+1)!}{(2m+4n-2i+1)!}\frac{(2m+4n-4i+1)}{(m+2n-2i)(m+2n-2i+1)}.
\end{multline}
Note that the right hand side of \eqref{eqn:mainidentity} is equal to
\begin{align}\label{eqn:7F6}
&\frac{(2n-1)!((m+2n)!)^2 (2m+2n+1)!}{((n-1)!)^2((m+n)!)^2 (2m+4n)!(m+2n)(m+2n+1)}\\\notag
&\times \sum_{i=0}^{n-1}\frac{\left( -m-2n-\frac{1}{2}, -\frac{m}{2}-n+\frac{3}{4}, \frac{1}{2}, -m-n, -\frac{m}{2}-n-\frac{1}{2}, -\frac{m}{2}-n, 1-n \right)_i}
{\left(1, -\frac{m}{2}-n-\frac{1}{4}, -m-2n, \frac{1}{2}-n, -\frac{m}{2}-n+1, -\frac{m}{2}-n+\frac{1}{2} , -m-n-\frac{1}{2} \right)_i}.
\end{align}
To compute the hypergeometric sum in \eqref{eqn:7F6}, we observe that it can
also be obtained from the $q\rightarrow 1$ limit of the
$(a,b,c,d,e, n)\mapsto (q^{-m-2n-\frac{1}{2}}, q^{\frac{1}{2}}, q^{-m-n}, q^{-\frac{m}{2}-n-\frac{1}{2}}, q^{-\frac{m}{2}-n}, n-1) $
special case of the basic hypergeometric sum
\begin{equation}\label{eqn:Watson0}
{}_8\phi_7\!\left[\begin{matrix}a,\,qa^{\frac 12},-qa^{\frac 12},b,c,d,e,q^{-n}\\
a^{\frac 12},-a^{\frac 12},aq/b,aq/c,aq/d,aq/e,aq^{n+1}\end{matrix};
q,\frac{a^2q^{n+2}}{bcde}\right].
\end{equation}
Now we analyze the ${}_8\phi_7$ series in \eqref{eqn:Watson0}.
First we use the well-known Watson transformation (cf.\ \cite[Appendix (III.18)]{GR}):
\begin{align}\label{eq:watson}
&{}_8\phi_7\!\left[\begin{matrix}a,\,qa^{\frac 12},-qa^{\frac 12},b,c,d,e,q^{-n}\\
a^{\frac 12},-a^{\frac 12},aq/b,aq/c,aq/d,aq/e,aq^{n+1}\end{matrix};
q,\frac{a^2q^{n+2}}{bcde}\right]\notag\\
&=\frac{(aq,aq/de;q)_n}{(aq/d,aq/e;q)_n}
{}_4\phi_3\!\left[\begin{matrix}aq/bc,d,e,q^{-n}\\
aq/b,aq/c,deq^{-n}/a\end{matrix};q,q\right].
\end{align}
 We also use the following summation (cf.\ \cite[Exercise 3.34]{GR})
\begin{equation}\label{eq:43}
{}_4\phi_3\!\left[\begin{matrix}q^{-2n},c^2,a,aq\\
a^2q^2,cq^{-n},cq^{1-n}\end{matrix};q^2,q^2\right]
=\frac{(-q,qa/c;q)_n}{(-aq,q/c;q)_n}.
\end{equation}

If we take \eqref{eq:watson}, replace $q$ by $q^2$, and specialize
$b=a/d^2$, $c=a^2q^{2+2n}/d^2$, the ${}_4\phi_3$ series on
the right-hand side can be simplified by the $(a,c)\mapsto (d,d^2q^{-n/a})$
case of \eqref{eq:43}. As a consequence, we obtain the following summation
(which is of interest by itself):
\begin{align}\label{eq:new}
&{}_8\phi_7\!\left[\begin{matrix}a,\,q^2a^{\frac 12},-q^2a^{\frac 12},
a/d^2,a^2q^{2+2n}/d^2,d,dq,q^{-2n}\\
a^{\frac 12},-a^{\frac 12},d^2q^2,d^2q^{-2n}/a,aq^2/d,aq/d,aq^{2n+2}\end{matrix};q^2,\frac{d^2q}a\right]\notag\\
&=\frac{(-q,aq/d^2;q)_n(aq^2;q^2)_n}{(-dq,aq/d;q)_n(aq^2/d^2;q^2)_n}.
\end{align}
We now apply the substitution $(a,d,n)\mapsto(q^{-1-2m-4n},q^{-1-m-2n},n-1)$ to \eqref{eq:new} and take the limit $q\rightarrow 1$ to obtain
\begin{align*}
&\sum_{i=0}^{n-1}\frac{\left( -m-2n-\frac{1}{2}, -\frac{m}{2}-n+\frac{3}{4}, \frac{1}{2}, -m-n, -\frac{m}{2}-n-\frac{1}{2}, -\frac{m}{2}-n, 1-n \right)_i}
{\left(1, -\frac{m}{2}-n-\frac{1}{4}, -m-2n, \frac{1}{2}-n, -\frac{m}{2}-n+1, -\frac{m}{2}-n+\frac{1}{2} , -m-n-\frac{1}{2} \right)_i}\\
&= \frac{(n-1)!n!((m+n)!)^2(2m+4n-1)!}{(2n-1)!((m+2n-1)!)^2 (2m+2n+1)!}.
\end{align*}
Using this result in \eqref{eqn:7F6} and simplifying the expression
proves the identity \eqref{eqn:mainidentity}.
\end{proof}

%----------------------------------------------------------------------------------------------------------

\section{Down degrees for shifted Young diagrams}
\label{sec:sum-down-degrees}

%----------------------------------------------------------------------------------------------------------

In this section we compute the expectation $\EE(X)$ for the interval
$[\emptyset,\lambda]$ when $\lambda$ is a trapezoidal shifted Young diagram. This
completes our proof of Theorem~\ref{thm:trapezoidal}. We first give a general
way to express $\EE(X)$ in terms of the number of shifted Young diagrams
contained in a given shifted Young diagram.

Let $\lambda$ be a shifted Young diagram. We denote by $R(\lambda)$ the number
of shifted Young diagrams $\mu\subseteq\lambda$. Define $R^{(+1)}(\lambda)$ to
be the sum of $\ddeg(\mu)$ for all shifted Young diagrams $\mu\subseteq\lambda$.
Equivalently, $R^{(+1)}(\lambda)$ is the number of pairs $(\mu,(i,j))$ of
$\mu\subseteq\lambda$ and an inner corner $(i,j)$ of $\mu$. By definition, the
expectation $\EE(X)$ for the interval $[\emptyset,\lambda]$ is given by
\[
\EE(X) = \frac{R^{(+1)}(\lambda)}{R(\lambda)}.
\]

The \emph{border} $\BB(\lambda)$ of $\lambda$ is the set of cells $(i,j)$ in
$\lambda$ such that $(i+1,j+1)$ is not in $\lambda$. See Figure~\ref{fig:border}
for an example.

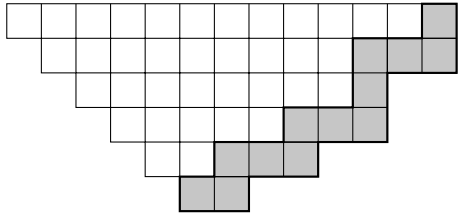
\begin{figure}[ht]
$$
\begin{tikzpicture}[scale=.46]
\draw[fill=gray!45] (5,0) rectangle (7,1);
\draw[fill=gray!45] (6,1) rectangle (9,2);
\draw[fill=gray!45] (8,2) rectangle (11,3);
\draw[fill=gray!45] (10,3) rectangle (11,5);
\draw[fill=gray!45] (11,4) rectangle (13,5);
\draw[fill=gray!45] (12,5) rectangle (13,6);
\draw (0,6)--(0,5)--(1,5)--(1,4)--(2,4)--(2,3)--(3,3)--(3,2)--(4,2)--(4,1)--(5,1)--(5,0)--(7,0)--(7,1)--(9,1)--(9,2)--(11,2)--(11,4)--(13,4)--(13,6)--(0,6)--cycle;
\draw (1,5)--(13,5)
	(2,4)--(11,4)
	(3,3)--(11,3)
	(4,2)--(9,2)
	(5,1)--(7,1)
	(1,5)--(1,6)
	(2,4)--(2,6)
	(3,3)--(3,6)
	(4,2)--(4,6)
	(5,1)--(5,6)
	(6,0)--(6,6)
	(7,1)--(7,6)
	(8,1)--(8,6)
	(9,2)--(9,6)
	(10,2)--(10,6)
	(11,4)--(11,6)
	(12,4)--(12,6);
\draw[thick] (5,0)--(7,0)--(7,1)--(9,1)--(9,2)--(11,2)--(11,4)--(13,4)--(13,6)--(12,6)--(12,5)--(10,5)--(10,3)--(8,3)--(8,2)--(6,2)--(6,1)--(5,1)--(5,0)--cycle;
\end{tikzpicture}
$$
  \caption{The border $\BB(\lambda)$ for a shifted Young diagram $\lambda$.}
  \label{fig:border}
\end{figure}

Suppose that
$\ell(\lambda)= n$. For $x=(i,j)\in\BB(\lambda)$, define
\[
\lambda(x) =
\begin{cases}
 (\lambda_1-1, \dots, \lambda_{n-1}-1), & \mbox{if $(i,j)=(n,n)$,}\\
 (\lambda_1-2, \dots, \lambda_{i-1}-2, \widetilde{\lambda}_{i+1},\dots,
 \tilde{\lambda}_{n}),& \mbox{otherwise,}
\end{cases}
\]
where
\[
\widetilde{\lambda}_t=
  \begin{cases}
    \lambda_t-1, & \mbox{if $\lambda_t+t-1=j$,}\\
    \lambda_t, & \mbox{otherwise.}
  \end{cases}
\]
Pictorially, if $x\in\BB(\lambda)\setminus\{(n,n)\}$, then
$\lambda(x)$ is obtained from $\lambda$ by removing the shaded region
and attaching the remaining two connected regions as shown in Figure~\ref{fig:la(x)}.

\begin{figure}[ht]
$$
\begin{tikzpicture}[scale=.46]
\draw[fill=gray!40] (0,5) rectangle (5,6);
\draw[fill=gray!40] (5,5)--(8,2)--(9,2)--(9,3)--(6,6)--(5,6)--(5,5);
\draw[fill=gray!40] (8,1) rectangle (9,2);
\draw[fill=gray!40] (9,2) rectangle (11,3);
\draw[very thick] (8,2)--(9,2)--(9,3)--(8,3)--(8,2)--cycle;
\draw[dashed, very thick] (1,5)--(5,5)--(8,2)--(8,1)
					(6,6)--(9,3)--(11,3);
\draw (0,6)--(0,5)--(1,5)--(1,4)--(2,4)--(2,3)--(3,3)--(3,2)--(4,2)--(4,1)--(5,1)--(5,0)--(7,0)--(7,1)--(9,1)--(9,2)--(11,2)--(11,4)--(13,4)--(13,6)--(0,6)--cycle;
\draw (1,5)--(13,5)
	(2,4)--(11,4)
	(3,3)--(11,3)
	(4,2)--(9,2)
	(5,1)--(7,1)
	(1,5)--(1,6)
	(2,4)--(2,6)
	(3,3)--(3,6)
	(4,2)--(4,6)
	(5,1)--(5,6)
	(6,0)--(6,6)
	(7,1)--(7,6)
	(8,1)--(8,6)
	(9,2)--(9,6)
	(10,2)--(10,6)
	(11,4)--(11,6)
	(12,4)--(12,6);
 \node[] at (8.5,2.5) {$x$};
 \node[] at (4.5, 2.5) {$A$};
  \node[] at (10.5, 4.5) {$B$};
  \node[] at (1,1.5) {$\lambda$};
 \end{tikzpicture}\qquad\qquad
 \begin{tikzpicture}[scale=.47]
\draw (0,5)--(0,4)--(1,4)--(1,3)--(2,3)--(2,2)--(3,2)--(3,1)--(4,1)--(4,0)--(6,0)--(6,1)--(7,1)--(7,2)--(9,2)--(9,3)--(11,3)--(11,5)--(0,5)--cycle; 
\draw (1,4)--(11,4)
	(2,3)--(9,3)
	(3,2)--(7,2)
	(4,1)--(6,1)
	(1,4)--(1,5)
	(2,3)--(2,5)
	(3,2)--(3,5)
	(4,1)--(4,5)
	(5,0)--(5,5)
	(6,1)--(6,5)
	(7,2)--(7,5)
	(8,2)--(8,5)
	(9,3)--(9,5)
	(10,3)--(10,5);
  \draw[dashed, very thick] (0,5)--(4,5)--(7,2)
  				(7,1)--(7,2)--(9,2);
  \node[] at (3.5, 2.5) {$A$};
  \node[] at (8.5, 3.5) {$B$};
   \node[] at (1,1) {$\lambda(x)$};
  \end{tikzpicture}
$$
  \caption{The shifted Young diagram $\lambda$ and $x\in \BB(\lambda)$ on the
    left and the diagram $\lambda(x)$ on the right.}
  \label{fig:la(x)}
\end{figure}
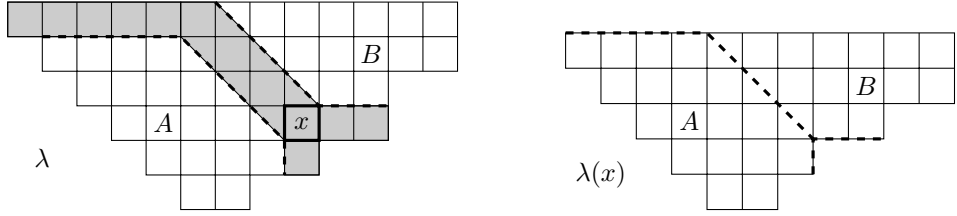

The following proposition allows us to write $R^{(+1)}(\lambda)$ as a sum of
$R(\mu)$'s.

\begin{prop}\label{prop:R+1}
  Let $\lambda$ be a shifted Young diagram. Then
\[
R^{(+1)}(\lambda) = \sum_{x\in \BB(\lambda)} R(\lambda(x)).
\]
\end{prop}
\begin{proof}
  By definition $R^{(+1)}(\lambda)$ is the number of pairs $(\mu,(i,j))$ of a
  shifted Young diagram $\mu\subseteq \lambda$ and an inner corner $c=(i,j)$ of $\mu$.
  Let $x$ be the cell such that $x=(i+t,j+t)\in \lambda$ for some $t\ge0$ with
  $(i+t+1,j+t+1)\not\in \lambda$. Then $x\in\BB(\lambda)$.
Define
  \[
    \nu=
    \begin{cases}
      (\mu_1-1,\mu_2-1,\dots,\mu_{i-1}-1), &
      \mbox{if $x=(n,n)$,}\\
      (\mu_1-2,\mu_2-2,\dots,\mu_{i-1}-2, \mu_{i+1},\mu_{i+2},\dots,\mu_k), &
      \mbox{if $x\ne(n,n)$,}\\
    \end{cases}
  \]
  where $k=\ell(\mu)$. See Figure~\ref{fig:nu1} and Figure~\ref{fig:nu2} for an
  example of $\nu$ for the cases $x\ne(n,n)$ and $x=(n,n)$, respectively. By the
  construction we have $\nu\in R(\lambda(x))$. It is not hard to see that the
  map $(\mu,(i,j))\mapsto (x,\nu)$ is a bijection from $R^{(+1)}(\lambda)$ to
  $\cup_{x\in\BB(\lambda)} R(\lambda(x))$. This proves the desired identity.
\end{proof}

\begin{figure}[ht]
$$
\begin{tikzpicture}[scale=.46]
\draw[fill=gray!30] (0,5) rectangle (5,6);
\draw[fill=gray!30] (5,5)--(8,2)--(9,2)--(9,3)--(6,6)--(5,6)--(5,5);
\draw[fill=gray!30] (8,1) rectangle (9,2);
\draw[fill=gray!30] (9,2) rectangle (11,3);
\draw[very thick] (8,2)--(9,2)--(9,3)--(8,3)--(8,2)--cycle
                            (6,4)--(7,4)--(7,5)--(6,5)--(6,4)--cycle;
\draw[dashed, very thick] (1,5)--(5,5)--(8,2)--(8,1)
					(6,6)--(9,3)--(11,3);
\draw (0,6)--(0,5)--(1,5)--(1,4)--(2,4)--(2,3)--(3,3)--(3,2)--(4,2)--(4,1)--(5,1)--(5,0)--(7,0)--(7,1)--(9,1)--(9,2)--(11,2)--(11,4)--(13,4)--(13,6)--(0,6)--cycle;
\draw (1,5)--(13,5)
	(2,4)--(11,4)
	(3,3)--(11,3)
	(4,2)--(9,2)
	(5,1)--(7,1)
	(1,5)--(1,6)
	(2,4)--(2,6)
	(3,3)--(3,6)
	(4,2)--(4,6)
	(5,1)--(5,6)
	(6,0)--(6,6)
	(7,1)--(7,6)
	(8,1)--(8,6)
	(9,2)--(9,6)
	(10,2)--(10,6)
	(11,4)--(11,6)
	(12,4)--(12,6);
\draw[very thick, color=blue] (0,6)--(0,5)--(1,5)--(1,4)--(2,4)--(2,3)--(3,3)--(3,2)--(4,2)--(4,1)--(5,1)--(5,3)--(6,3)--(6,4)--(7,4)--(7,5)--(10,5)--(10,6)--(0,6)--cycle;
\node[] at (1.5, 2) {\color{blue}$\mu$};
\node[] at (12.3,2.5) {$\lambda$};
 \node[] at (8.5,2.5) {$x$};
  \node[] at (6.5,4.5) {$c$};
 \end{tikzpicture}\qquad\qquad
 \begin{tikzpicture}[scale=.47]
\draw (0,5)--(0,4)--(1,4)--(1,3)--(2,3)--(2,2)--(3,2)--(3,1)--(4,1)--(4,0)--(6,0)--(6,1)--(7,1)--(7,2)--(9,2)--(9,3)--(11,3)--(11,5)--(0,5)--cycle; 
\draw (1,4)--(11,4)
	(2,3)--(9,3)
	(3,2)--(7,2)
	(4,1)--(6,1)
	(1,4)--(1,5)
	(2,3)--(2,5)
	(3,2)--(3,5)
	(4,1)--(4,5)
	(5,0)--(5,5)
	(6,1)--(6,5)
	(7,2)--(7,5)
	(8,2)--(8,5)
	(9,3)--(9,5)
	(10,3)--(10,5);
  \draw[dashed, very thick] (4,5)--(7,2)
  				(7,1)--(7,2)--(9,2);
 \draw[very thick, color=blue] (0,5)--(0,4)--(1,4)--(1,3)--(2,3)--(2,2)--(3,2)--(3,1)--(4,1)--(4,3)--(5,3)--(5,4)--(8,4)--(8,5)--(0,5)--cycle;
\node[] at (1, 1.5) {\color{blue}$\nu$};
\node[] at (10, 1) {$\lambda(x)$};
  \end{tikzpicture}
$$
  \caption{The diagram $\nu$ for $x\ne(n,n)$.}
  \label{fig:nu1}
\end{figure}

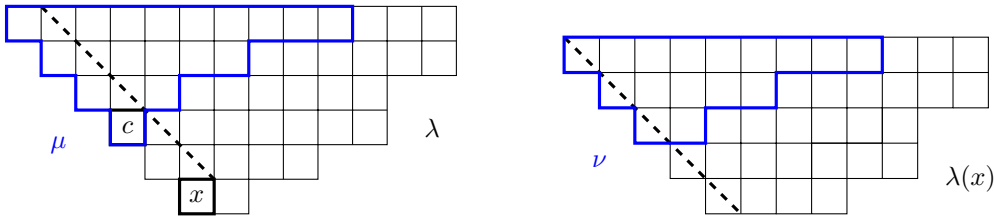
\begin{figure}[ht]
$$
\begin{tikzpicture}[scale=.46]
\draw[very thick] (5,0)--(6,0)--(6,1)--(5,1)--(5,0)--cycle
                            (3,2)--(4,2)--(4,3)--(3,3)--(3,2)--cycle;
\draw[dashed, very thick] (1,6)--(6,1);
\draw (0,6)--(0,5)--(1,5)--(1,4)--(2,4)--(2,3)--(3,3)--(3,2)--(4,2)--(4,1)--(5,1)--(5,0)--(7,0)--(7,1)--(9,1)--(9,2)--(11,2)--(11,4)--(13,4)--(13,6)--(0,6)--cycle;
\draw (1,5)--(13,5)
	(2,4)--(11,4)
	(3,3)--(11,3)
	(4,2)--(9,2)
	(5,1)--(7,1)
	(1,5)--(1,6)
	(2,4)--(2,6)
	(3,3)--(3,6)
	(4,2)--(4,6)
	(5,1)--(5,6)
	(6,0)--(6,6)
	(7,1)--(7,6)
	(8,1)--(8,6)
	(9,2)--(9,6)
	(10,2)--(10,6)
	(11,4)--(11,6)
	(12,4)--(12,6);
\draw[very thick, color=blue] (0,6)--(0,5)--(1,5)--(1,4)--(2,4)--(2,3)--(3,3)--(3,2)--(4,2)--(4,3)--(5,3)--(5,4)--(7,4)--(7,5)--(10,5)--(10,6)--(0,6)--cycle;
\node[] at (1.5, 2) {\color{blue}$\mu$};
\node[] at (12.3,2.5) {$\lambda$};
 \node[] at (5.5,.5) {$x$};
  \node[] at (3.5,2.5) {$c$};
 \end{tikzpicture}\qquad\qquad
 \begin{tikzpicture}[scale=.47]
\draw (0,5)--(0,4)--(1,4)--(1,3)--(2,3)--(2,2)--(3,2)--(3,1)--(4,1)--(4,0)--(6,0)--(6,1)--(7,1)--(7,2)--(9,2)--(9,3)--(11,3)--(11,5)--(0,5)--cycle; 
\draw (1,4)--(11,4)
	(2,3)--(9,3)
	(3,2)--(10,2)
	(4,1)--(6,1)
	(1,4)--(1,5)
	(2,3)--(2,5)
	(3,2)--(3,5)
	(4,1)--(4,5)
	(5,0)--(5,5)
	(6,1)--(6,5)
	(7,0)--(7,5)
	(8,1)--(8,5)
	(9,1)--(9,5)
	(10,3)--(10,5)
	(11,5)--(12,5)--(12,3)--(11,3)
	(11,4)--(12,4)
	(10,3)--(10,1)--(7,1)
	(6,0)--(8,0)--(8,1);
  \draw[dashed, very thick] (0,5)--(5,0);
 \draw[very thick, color=blue] (0,5)--(0,4)--(1,4)--(1,3)--(2,3)--(2,2)--(4,2)--(4,3)--(6,3)--(6,4)--(9,4)--(9,5)--(0,5)--cycle;
\node[] at (1, 1.5) {\color{blue}$\nu$};
\node[] at (11.5, 1) {$\lambda(x)$};
  \end{tikzpicture}
$$
  \caption{The diagram $\nu$ for $x=(n,n)$.}
  \label{fig:nu2}
\end{figure}

In the next two lemmas we find simple formulas for $R(\lambda)$ and
$R^{(+1)}(\lambda)$ for a trapezoidal shifted Young diagram $\lambda$.

\begin{lem}\label{lem:R(la)}
  Let $\lambda=(N,N-2,\dots,N-2n+2)$. Then
\[
R(\lambda) = \binom{N+1}{n}.
\]
\end{lem}
\begin{proof}
  Let us embed the shifted Young diagram of $\lambda$ in $\mathbb{Z}^2$ so that
  each cell is a unit square and the top left corner of $\lambda$ is at $(0,0)$
  as shown in Figure~\ref{fig:embed}.

\begin{figure}[ht]
$$
\begin{tikzpicture}[scale=.48]
\draw (0,3)--(0,2)--(1,2)--(1,1)--(2,1)--(2,0)--(5,0)--(5,1)--(6,1)--(6,2)--(7,2)--(7,3)--(0,3)--cycle
	(1,2)--(6,2)
	(2,1)--(5,1)
	(1,2)--(1,3)
	(2,1)--(2,3)
	(3,0)--(3,3)
	(4,0)--(4,3)
	(5,1)--(5,3)
	(6,2)--(6,3);
\draw[dashed] (4,-1)--(8.5,3.5)
				(7,3)--(8,3);
\filldraw [black] (0,3) circle (3.5pt)
			(3,0) circle (3.5pt)
			(8,3) circle (3.5pt);
\node[] at (0, 3.6) {\small $(0,0)$};
\node[] at (9.5, 2.8) {\small (N+1,0)};
\node[] at (3,-.6) {\small $(n,-n)$};
\node[] at (-1.3, 1.3) {$\lambda=$};
\node[] at (8, .4) {\small $y=x-N-1$};
\end{tikzpicture}
$$
  \caption{Embedding $\lambda$ in $\ZZ^2$.}
  \label{fig:embed}
\end{figure}
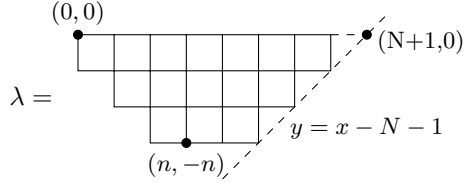

Each shifted Young diagram $\mu\subseteq\lambda$ can be identified with a
lattice path from $(j,-j)$ to $(N+1,0)$ for some $0\le j\le n$ consisting of
north and east steps that never goes below the line $y=x-N-1$. For example, if
$N=7, n=3$ so that $\lambda=(7,5,3)$, then the shifted shape $\mu=(7,3)\subseteq
\lambda$ is identified with the path shown in Figure~\ref{fig:lattice}.

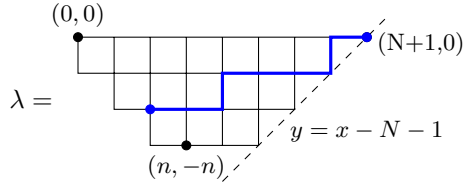
\begin{figure}[ht]
$$
\begin{tikzpicture}[scale=.48]
\draw (0,3)--(0,2)--(1,2)--(1,1)--(2,1)--(2,0)--(5,0)--(5,1)--(6,1)--(6,2)--(7,2)--(7,3)--(0,3)--cycle
	(1,2)--(6,2)
	(2,1)--(5,1)
	(1,2)--(1,3)
	(2,1)--(2,3)
	(3,0)--(3,3)
	(4,0)--(4,3)
	(5,1)--(5,3)
	(6,2)--(6,3);
\draw[dashed] (4,-1)--(8.5,3.5)
				(7,3)--(8,3);
\filldraw [black] (0,3) circle (3.5pt)
			(3,0) circle (3.5pt)
			(8,3) circle (3.5pt);
\node[] at (0, 3.6) {\small $(0,0)$};
\node[] at (9.5, 2.8) {\small (N+1,0)};
\node[] at (3,-.6) {\small $(n,-n)$};
\node[] at (-1.3, 1.3) {$\lambda=$};
\node[] at (8, .4) {\small $y=x-N-1$};
 \draw[very thick, color=blue] (2,1)--(4,1)--(4,2)--(7,2)--(7,3)--(8,3);
\filldraw [blue] (8,3) circle (3.5pt)
                       (2,1) circle (3.5pt);
\end{tikzpicture}
$$
  \caption{The lattice path representing $\mu=(7,3)$.}
  \label{fig:lattice}
\end{figure}

By the standard reflection method, one can see that, for a fixed $0\le j\le n$,
the number of such paths equals $\binom{N+1}{j}-\binom{N+1}{j-1}$. Therefore the
total number of shifted Young diagrams $\mu$ contained in $\lambda$ is
\[
\sum_{j=0}^n \left(  \binom{N+1}{j}-\binom{N+1}{j-1}\right) = \binom{N+1}{n},
\]
as the sum telescopes.
\end{proof}

\begin{lem}\label{lem:R+(la)}
  Let $\lambda=(N,N-2,\dots,N-2n+2)$. Then
\[
R^{(+1)}(\lambda) =\frac{|\lambda|}{N+1} \binom{N+1}{n}.
\]
\end{lem}
\begin{proof}
  Let $T_{a,b}$ denote the shifted Young diagram $(a,a-2,\dots,a-2b+2)$.
  By Proposition~\ref{prop:R+1},
  \begin{equation}
    \label{eq:4}
R^{(+1)}(T_{N,n}) = \sum_{x\in \BB(T_{N,n})} R(T_{N,n}(x)).
  \end{equation}

  Let $x\in\BB(T_{N,n})$. It is straightforward to check that
$R(T_{N,n}(x))$ is given as follows. 
\begin{itemize}
\item If $x=(n,n)$,
  \[
R(T_{N,n}(x)) = T_{N-1,n-1}.
  \]  
\item  If $x\ne (n,n)$ and $x$ is in the $n$th row,
  \[
R(T_{N,n}(x)) = T_{N-2,n-1}.
  \]  
\item If $x$ is the rightmost cell in the $i$th row ($1\le i\le n-1$),
  \[
R(T_{N,n}(x)) = T_{N-2,n-1}.
  \]  
\item If $x$ is the second rightmost cell in the $i$th row ($1\le i\le n-1$),
  \[
R(T_{N,n}(x)) =T_{N-2,n-1}\setminus\{(i,N-2i)\}.
  \]  
\end{itemize}
Therefore, by \eqref{eq:4},
\begin{align}
  \label{eq:3}
  R^{(+1)}(T_{N,n}(x)) = R(T_{N-1,n-1}) +
  ((N-2n+1)+(n-1))R(T_{N-2,n-1})& \\\notag
  +\sum_{i=1}^{n-1} R(T_{N-2,n-1}\setminus\{(i,N-2i)\})&.
\end{align}

Let $1\le i\le n-1$. Then 
\begin{equation}
  \label{eq:6}
  R(T_{N-2,n-1}\setminus\{(i,N-2i)\}) =
  R(T_{N-2,n-1}) - t(i),
\end{equation}
where $t(i)$ is the number of shifted Young diagrams $\mu\subseteq T_{N-2,n-1}$
containing the cell $(i,N-2i)$. Suppose $\mu$ is such a shifted Young diagram.
Then $\mu$ is determined by the two sub-diagrams $\alpha$ and $\beta$ of $\mu$,
where $\alpha$ is the set of cells of $\mu$ in columns $N-2,N-3,\dots,N-i$ and
$\beta$ is the set of cells of $\mu$ in rows $i+1,i+2,\dots,n-1$, see Figure~\ref{fig:decomp}.

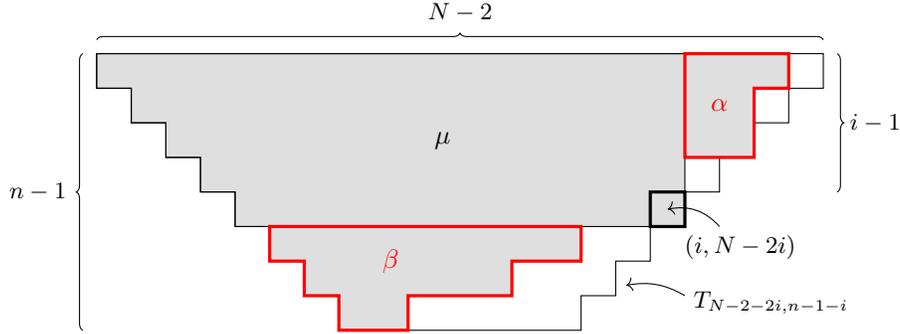
\begin{figure}[ht]
$$
\begin{tikzpicture}[scale=.46]
 \draw [snake=brace] (0, 8.4)--(21,8.4)
 				(-.4, 0)--(-.4, 8)
				(21.4, 8)--(21.4, 4);
\draw[fill=gray!25] (0,8)--(0,7)--(1,7)--(1,6)--(2,6)--(2,5)--(3,5)--(3,4)--(4,4)--(4,3)--(5,3)--(5,2)--(6,2)--(6,1)--(7,1)--(7,0)--(9,0)--(9,1)--(12,1)--(12,2)--(14,2)--(14,3)--(17,3)--(17,5)--(19,5)--(19,7)--(20,7)--(20,8)--(0,8)--cycle;
\draw (0,8)--(0,7)--(1,7)--(1,6)--(2,6)--(2,5)--(3,5)--(3,4)--(4,4)--(4,3)--(5,3)--(5,2)--(6,2)--(6,1)--(7,1)--(7,0)--(14,0)--(14,1)--(15,1)--(15,2)--(16,2)--(16,3)--(17,3)--(17,4)--(18,4)--(18,5)--(19,5)--(19,6)--(20,6)--(20,7)--(21,7)--(21,8)--(0,8)--cycle;
\draw[very thick, color=red] (5,3)--(5,2)--(6,2)--(6,1)--(7,1)--(7,0)--(9,0)--(9,1)--(12,1)--(12,2)--(14,2)--(14,3)--(5,3)--cycle;
\draw[very thick, color=red] (17,5)--(19,5)--(19,7)--(20,7)--(20,8)--(17,8)--(17,6)--cycle;
\draw[very thick] (16,3)--(17,3)--(17,4)--(16,4)--(16,3)--cycle;
\node[] at (8.5, 2) {\color{red} $\beta$};
\node[] at (18, 6.5) {\color{red} $\alpha$};
\node[] at (-1.7, 4) {\small $n-1$};
\node[] at (10.5, 9.2) {\small $N-2$};
\draw[->] (18, 3)to [bend right=30] (16.5, 3.5);
\node[] at (18.6, 2.4) {\small $(i, N-2i)$};
\draw[->] (17,1)to [bend right=20] (15.3, 1.3);
\node[] at (19.5,.8) {\small $T_{N-2-2i, n-1-i}$};
\node[] at (22.5, 6) {\small $i-1$};
\node[] at (10,5.5) {$\mu$};
\end{tikzpicture}
$$

  \caption{The diagram $\mu$ is determined by $\alpha$ and $\beta$.}
  \label{fig:decomp}
\end{figure}

Then one can
regard $\alpha$ as a Young diagram contained in the Young diagram
$(i-1,i-2,\dots,1)$ and $\beta$ as a shifted Young diagram contained in
$T_{N-2-2i,n-1-i}$. It is well known that the number of such $\alpha$ is given by
the Catalan number $\frac{1}{i+1}\binom{2i}{i}$. This argument shows that
\begin{equation}
  \label{eq:5}
t(i) = \frac{1}{i+1}\binom{2i}{i} R(T_{N-2-2i,n-1-i}).
\end{equation}
By \eqref{eq:3}, \eqref{eq:6}, \eqref{eq:5}, and Lemma~\ref{lem:R(la)}, we have
\begin{align*}
  R^{(+1)}(T_{N,n}(x)) &= R(T_{N-1,n-1}) + (N-1)R(T_{N-2,n-1}) 
                         -\sum_{i=1}^{n-1}\frac{1}{i+1}\binom{2i}{i} R(T_{N-2-2i,n-1-i}),\\
                       &= \binom{N}{n-1} + (N-1)\binom{N-1}{n-1}
                         -\sum_{i=1}^{n-1}\frac{1}{i+1}\binom{2i}{i}\binom{N-1-2i}{n-1-i} .
\end{align*}
Therefore the identity we need to show is
\begin{equation}
  \label{eq:8}
\frac{n(N-n+1)}{N+1} \binom{N+1}{n}  = \binom{N}{n-1} + N\binom{N-1}{n-1}
-\sum_{i=0}^{n-1}\frac{1}{i+1}\binom{2i}{i}\binom{N-1-2i}{n-1-i}
\end{equation}
which can be simplified to 
\begin{equation}\label{eq:9}
\frac{N}{N-n+1}=\sum_{i=0}^{n-1}\frac{\left( \frac{1}{2}, 1, n-N, 1-n\right)_i}{\left(1,  2, \frac{1-N}{2}, 1-\frac{N}{2}\right)_i}.
\end{equation}
Equation \eqref{eq:9} can be proved by applying the Bailey formula \cite[p. 512, (c)]{Bailey1929}
$$
{}_4 F_3\!\left[\begin{matrix} \frac{\mathsf{a}}{2}, \frac{\mathsf{a}+1}{2}, \mathsf{b}+\mathsf{n}, -\mathsf{n}\\
\frac{\mathsf{b}}{2}, \frac{\mathsf{b}+1}{2}, \mathsf{a}+1 \end{matrix};1\right]
=\frac{(\mathsf{b}-\mathsf{a})_\mathsf{n}}{(\mathsf{b})_\mathsf{n}}
$$
with $\mathsf{a}\mapsto 1$, $\mathsf{b}\mapsto 1-N$ and $\mathsf{n}\mapsto n-1$.
\end{proof}

Finally we can evaluate $\EE(X)$ for the lower interval of the shifted Young
poset below a trapezoidal shifted Young diagram.

\begin{thm}\label{thm:EX}
  Let $\lambda$ be a trapezoidal shifted Young diagram. Then
  the expectation $\EE(X)$ for the interval $[\emptyset,\lambda]$ is equal to
  \[
\EE(X)= \frac{|\lambda|}{\lambda_1+1}.
  \]
\end{thm}
\begin{proof}
  Since $\lambda$ is trapezoidal, we can write $\lambda=(N,N-2,\dots,N-2n+2)$ for
  some nonnegative integers $N,n$. Then we have
  \[
    \EE(X) = \frac{R^{(+1)}(\lambda)}{R(\lambda)}
    = \frac{|\lambda|}{N+1}= \frac{|\lambda|}{\lambda_1+1},
  \]
   by Lemmas~\ref{lem:R(la)} and
  \ref{lem:R+(la)}.
\end{proof}

Theorem~\ref{thm:trapezoidal} in the introduction then follows from
Theorems~\ref{thm:g+1 trapezoidal} and \ref{thm:EX}.

\begin{remark}\label{rem:antichain}
  We note that Theorem~\ref{thm:EX} can also be proved by combining known
  results in the literature as follows. This proof is due to Sam Hopkins
  (personal communication).

  Suppose $L$ is a distributive lattice. Then $L$ can be written as the poset
  $J(P)$ of order ideals of some poset $P$, see \cite[3.4.1~Theorem]{EC1}. One
  can show that the down-degree expectation $\EE_L(X)$ for $L$ with respect to
  the uniform distribution is the average size of an antichain in $P$, see
  \cite[Section~3.1]{Hopkins_2020}.

  Now let $L$ be the lower interval $[\emptyset,\lambda]$ for a trapezoidal
  shifted Young diagram $\lambda$. Then $L=J(P_\lambda)$, where $P_\lambda$
  denotes the poset whose elements are the cells in $\lambda$ and two elements
  $x,y\in P_\lambda$ satisfy $x< y$ if $x$ is weakly to the southeast of $y$ in
  $\lambda$. Then by the fact in the above paragraph, $\EE_L(X)$ is the average
  size of an antichain in $P_\lambda$. Stembridge
  \cite[Corollary~2.4]{Stembridge_1986} showed that if $\lambda$ is the
  trapezoidal shape $(m+n-1,m+n-3,\dots,m+n-(2m-1))$ and $\mu$ is the rectangle
  $(n^m)$, then, for each $k$, the number of antichains of size $k$ is equal to
  $\binom{m}{k}\binom{n}{k}$ for both $P_\lambda$ and $P_\mu$. In particular,
  the average size of an antichain in $P_\lambda$ is equal to that in $P_\mu$.
  This shows that $\EE_L(X)=\EE_{L'}(X)$, where $L'=J(P_\mu)$. Then $L'$ is
  isomorphic to the lower interval $[\emptyset,\mu]$ for $\mu=(n^m)$. Therefore,
  by Theorem~\ref{thm:CHHM},
\[
\EE_L(X)=\EE_{L'}(X)=\frac{mn}{m+n}=\frac{|\lambda|}{\lambda_1+1},
\]
which is Theorem~\ref{thm:EX}.
\end{remark}

%----------------------------------------------------------------------------------------------------------
\appendix
\section{An $\ma;q$-analogue of the expectation $\EE(X)$}
\label{sec:an-ma-q}
%----------------------------------------------------------------------------------------------------------

Our proof of Theorem~\ref{sec:proof} makes use of $q$-integrals and in
our applications we either utilize identities for basic hypergeometric series
(see in particular the proof of Theorem~\ref{thm:g+1 trapezoidal}) or use
summations which have $q$-analogues (see e.g. the proofs of
Theorems~\ref{thm:conj24} and \ref{thm:thm4.2}). It is thus natural to ask
whether $q$-analogues of the results proved in this paper exist.

Reiner, Tenner, and Yong \cite[Proposition~1.5]{RTY2018} considered the
$q$-analog $R(\lambda,q)=\sum_{\mu\subseteq\la}q^{|\mu|}$ of the number
$R(\lambda)$ of Young diagrams contained in $\lambda$ and found a recurrence
satisfied by $R(\la,q)$. Hopkins \cite[Section~3.3]{Hopkins_2020} considered
certain $q$-analogues of down-degree generating functions for $P$-partitions,
which contain the lower interval $[\emptyset,\lambda]$ as a special case.

Various possible $q$-analogues of $\EE(X)$ or $\EE(Y)$ are feasible. A good
$q$-analogue may also be accompanied by nice results, such as product formulas
extending those for ordinary enumeration.

It seems rather difficult to find a $q$-analogue of the down-degree expectation
(with respect to the maximal chain cardinality distribution) $\EE(Y)$ that would
allow closed form results in cases of interest. We propose an $\ma;q$-analogue,
where $\ma$ is an extra free parameter, of the simpler expectation (with respect
to the uniform distribution) $\EE(X)$ which conjecturally can be expressed as a
product for a suitably restricted class of lower intervals of Young's lattice.

We work with so-called $\ma;q$-weights.
These are important special cases of the elliptic weights that were originally
introduced by the second author in \cite{Schlosser2007} in a closely related
context and made further appearance in a series of papers devoted to
elliptic combinatorics (see \cite{BCK,SchlosserYoo2017b,SchlosserYoo2018}
and the references therein).
In some instances the enumeration with respect to elliptic weights does not
yield closed formulas, but the specialization to $\ma;q$-weights does. 
See in particular, \cite{SchlosserYoo2017c}, which shows how
the $\ma;q$-enumeration of rook configurations can be used to
obtain (or recover) summations for basic hypergeometric series.
A similar feature seems to apply here as well.

Let $\ma$ and $q$ be indeterminates.
For a non-negative integer $n$, define the $\ma;q$-weight by
\begin{subequations}
\begin{equation}\label{aqw}
W_{\ma;q}(n)=\frac{1-\ma q^{1+2n}}{1-\ma q}q^{-n},
\end{equation}
and the $\ma;q$-number by
\begin{equation}\label{aqn}
[n]_{\ma;q}:=\frac{(1-q^n)(1-\ma q^n)}{(1-q)(1-\ma q)}q^{1-n}.
\end{equation}
\end{subequations}
Clearly, $W_{\ma;q}(0)=1$.
It is easy to see that the sum of the $\ma;q$-weights telescope
to the $\ma;q$-numbers:
\begin{equation*}
\sum_{k=0}^{n-1}W_{\ma;q}(k)=[n]_{\ma;q}.
\end{equation*}
For $\ma\to\infty$, the $\ma;q$-weight in \eqref{aqw}
and the $\ma;q$-number in \eqref{aqn} 
reduce to the $q$-weight $q^n$ and to the standard $q$-number $[n]_q$,
respectively. (Further, for $\ma\to -1$ the $\ma;q$-number in \eqref{aqn}
reduces to the ``quantum number" $(q^n-q^{-n})/(q-q^{-1})$,
which is a $q$-analogue of $n$ that satisfies the symmetry
$q\leftrightarrow q^{-1}$.)

We now define an $\ma;q$-analogue of the down-degree expectation $\EE(X)$.
Recall that for each partition $x$, the number of cells of $x$ is denoted by
$|x|$.

\begin{defn}
  Let $\la$ be a fixed Young diagram. Write $w=\lambda_1$ and
  $\ell=\ell(\lambda)$ for the width and length of $\la$, respectively, and let
  $d=\gcd(w,\ell)$. For $x,y\in [\emptyset,\lambda]$ write $y\lessdot x$ if $y$
  is covered by $x$ in this lower interval. If $y$ is obtained from $x$ by
  deleting a cell in row $s$, define the weight $\wt(x,y)$ by
$$\wt(x,y) = W_{\ma q^{\frac{w+\ell}d};q}\bigg(\frac{w(s-1)+\ell(|x|-1)}d\bigg).$$
Define $\aqddeg(x)$ to be the sum of the weights $\wt(x,y)$ for all
$y\in[\emptyset,\lambda]$ satisfying $y\lessdot x$. Further, define the
$\ma;q$-weight generating function of $[\emptyset,\lambda]$ by
$$
R(\la|\,\ma;q)=\sum_{x\in [\emptyset,\lambda]} W_{\ma q^{\frac{w\ell}d};q}\bigg(\frac{\ell|x|}d\bigg).
$$
Now define $\EE_{\ma;q}(X)$ by
\begin{equation}\label{aqexp}
\EE_{\ma;q}(X)=
\frac{\sum_{x\in [\emptyset,\lambda]} \aqddeg(x)}
{R(\la|\,\ma;q)}.
\end{equation}
\end{defn}

Notice that for $\ma\to\infty$, the $\ma;q$-weight generating function
$R(\la|\,\ma;q)$ reduces to $R\big(\la,q^{\frac{\ell}d}\big)$, not to $R(\la,q)$.
With the above definitions the following conjecture is easy to state and takes
a symmetric form.
\begin{conj}\label{conjaq}
If $\la$ is a balanced partition (of any slope), then the product formula
\begin{equation}\label{aqexpf}
\EE_{\ma;q}(X)=\frac{\big[\frac {w\ell}d\big]_{\ma q^{\frac{w+\ell}d};q}}
{\big[\frac {w+\ell}d\big]_{\ma q^{\frac{w\ell}d};q}}
\end{equation}
holds, where $w=\lambda_1$, $\ell=\ell(\lambda)$, and $d=\gcd(w,\ell)$.
\end{conj}

It is not difficult to verify directly, by using the definition for the $\ma; q$-numbers,
that \eqref{aqexpf} can be alternatively written as
\begin{equation}\label{aqexpf2}
\EE_{\ma;q}(X)=
\frac{\big[\frac {w\ell}d\big]_{\ma q;q^{\frac{w+\ell}d}}}
{\big[\frac {w+\ell}d\big]_{\ma q; q^{\frac{w\ell}d}}}.
\end{equation}

It is interesting to see from \eqref{aqexpf} (or from \eqref{aqexpf2})
that for balanced partitions $\la$, the down-degree expectation $\EE_{\ma;q}(X)$
on $[\emptyset,\lambda]$ is the same as on the poset $[\emptyset,\lambda']$,
the lower interval poset related to the conjugate $\la'$ of $\la$.
This symmetry is not obvious from \eqref{aqexp} and an explanation of
this fact itself would be desirable.

Conjecture~\ref{conjaq} is easy to confirm in the cases of $\la$ consisting
of one row or of one column, by induction. In the case of $\la$ consisting
of one row (resp.\ one column), the numerator in \eqref{aqexp} simplifies
to the numerator in \eqref{aqexpf} (resp.\ \eqref{aqexpf2}), while the
denominator in \eqref{aqexp} then simplifies to the
denominator in \eqref{aqexpf} (resp.\ \eqref{aqexpf2}).

It is also easy to show that if $\la=(w^\ell)$ is a rectangular shape, the
denominator $R(\la|\,\ma;q)$ simplifies into a product as follows.
First of all, it is well-known (and corresponds to a classical result by MacMahon)
that
$$
\sum_{x\in [\emptyset,(w^\ell)]}q^{|x|}=\begin{bmatrix}w+\ell\\\ell\end{bmatrix}_q,
$$
where
$$
\begin{bmatrix}w+\ell\\\ell\end{bmatrix}_q=\frac{(q;q)_{w+\ell}}{(q;q)_w(q;q)_\ell}
$$
is the $q$-binomial coefficient. By definition one can easily see that
$$
\begin{bmatrix}w+\ell\\\ell\end{bmatrix}_{q^{-1}}=
\begin{bmatrix}w+\ell\\\ell\end{bmatrix}_q\,q^{-w\ell}.
$$
Now
\begin{align*}
\sum_{x\in [\emptyset,(w^\ell)]}W_{\ma;q}(|x|)&=
\sum_{x\in [\emptyset,(w^\ell)]}\frac{1-\ma q^{1+2|x|}}{1-\ma q}q^{-|x|}
=(1-\ma q)^{-1}\left(\sum_{x\in [\emptyset,(w^\ell)]}q^{-|x|}-
\ma q\sum_{x\in [\emptyset,(w^\ell)]}q^{|x|}\right)\\
&=(1-\ma q)^{-1}\left(\begin{bmatrix}w+\ell\\\ell\end{bmatrix}_{q^{-1}}-
\ma q\begin{bmatrix}w+\ell\\\ell\end{bmatrix}_{q}\right)
=\begin{bmatrix}w+\ell\\\ell\end{bmatrix}_q
\frac{1-\ma q^{1+w\ell}}{1-\ma q}q^{-w\ell}.
\end{align*}
This immediately implies
\begin{equation}
R\big((w^\ell)|\,\ma;q\big)=
\begin{bmatrix}w+\ell\\\ell\end{bmatrix}_{q^{\frac{\ell}d}}
\frac{1-\ma q^{1+\frac{w\ell(\ell+1)}d}}
{1-\ma q^{1+\frac{w\ell}d}}q^{-\frac{w\ell^2}d}.
\end{equation}

We finally give three concrete examples which illustrate cancellation of
\textit{non-linear} factors in the computation of $\EE_{\ma;q}(X)$. (At the same
time, they explain why the parameter $\ma$ in the numerator and denominator is
shifted differently. Without these shifts, the $\ma$-dependent factors would not
cancel each other.)

\begin{exam}
Let $\la=(4,2)$.
Then $(w,\ell,d)=(4,2,2)$ and
\begin{align*}
\sum_{x\in [\emptyset,\lambda]} \aqddeg(x)&=0+1+W_{\ma q^3;q}(1)
+2W_{\ma q^3;q}(2)+3W_{\ma q^3;q}(3)\\[-.7em]
&\quad\;+3W_{\ma q^3;q}(4)+3W_{\ma q^3;q}(5)
+2W_{\ma q^3;q}(6)+W_{\ma q^3;q}(7)\\
&=\frac{(1-q^4)(1+q+q^2+q^4-\ma q^{11}-\ma q^{13}-\ma q^{14}-\ma q^{15})}
{(1-q)(1-\ma q^4)}q^{-7},
\end{align*}
and
\begin{align*}
R(\la|\,\ma;q)&=
1+W_{\ma q^4;q}(1)+2W_{\ma q^4;q}(2)+2W_{\ma q^4;q}(3)
+3W_{\ma q^4;q}(4)%\\&\quad\;
+2W_{\ma q^4;q}(5)+W_{\ma q^4;q}(6)\\
&=\frac{(1-q^3)(1+q+q^2+q^4-\ma q^{11}-\ma q^{13}-\ma q^{14}-\ma q^{15})}
{(1-q)(1-\ma q^5)}q^{-6},
\end{align*}
so according to \eqref{aqexp},
\begin{equation*}
\EE_{\ma;q}(X)=
\frac{\sum_{x\in [\emptyset,\lambda]} \aqddeg(x)}
{R(\la|\,\ma;q)}=
\frac{(1-q^4)(1-\ma q^5)}{(1-q^3)(1-\ma q^4)}q^{-1}
=\frac{[4]_{\ma q^3;q}}
{[3]_{\ma q^4;q}},
\end{equation*}
which agrees with the $(w,\ell,d)=(4,2,2)$ case of \eqref{aqexpf}.
\end{exam}

\begin{exam}
Let $\la=(2,2,1,1)$.
Then $(w,\ell,d)=(2,4,2)$ and
\begin{align*}
\sum_{x\in [\emptyset,\lambda]} \aqddeg(x)&=0+1+W_{\ma q^3;q}(2)
+W_{\ma q^3;q}(3)+W_{\ma q^3;q}(4)
+W_{\ma q^3;q}(5)+2W_{\ma q^3;q}(6)
+W_{\ma q^3;q}(7)\\[-.7em]
&\quad\;
+2W_{\ma q^3;q}(8)+2W_{\ma q^3;q}(9)%\\&\quad\;
+W_{\ma q^3;q}(10)+2W_{\ma q^3;q}(11)+W_{\ma q^3;q}(13)
\\
&=\frac{(1-q^4)(1-q^6)
(1+q^2+q^4+q^8-\ma q^{17}-\ma q^{21}-\ma q^{23}-\ma q^{25})}
{(1-q^2)(1-q^3)(1-\ma q^4)}q^{-13},
\end{align*}
and
\begin{align*}
R(\la|\,\ma;q)&=
1+W_{\ma q^4;q}(2)+2W_{\ma q^4;q}(4)+2W_{\ma q^4;q}(6)
+3W_{\ma q^4;q}(8)%\\&\quad\;
+2W_{\ma q^4;q}(10)+W_{\ma q^4;q}(12)\\
&=\frac{(1-q^6)
(1+q^2+q^4+q^8-\ma q^{17}-\ma q^{21}-\ma q^{21}-\ma q^{25})}
{(1-q^2)(1-\ma q^5)}q^{-12},
\end{align*}
so according to \eqref{aqexp},
\begin{equation*}
\EE_{\ma;q}(X)=
\frac{\sum_{x\in [\emptyset,\lambda]} \aqddeg(x)}
{R(\la|\,\ma;q)}=
\frac{(1-q^4)(1-\ma q^5)}{(1-q^3)(1-\ma q^4)}q^{-1}
=\frac{[4]_{\ma q^3;q}}
{[3]_{\ma q^4;q}},
\end{equation*}
which agrees with the $(w,\ell,d)=(2,4,2)$ case of \eqref{aqexpf}.
\end{exam}

\begin{exam} 
Let $\la=(3,2,1)$.
Then $(w,\ell,d)=(3,3,3)$ and
\begin{align*}
\sum_{x\in [\emptyset,\lambda]} \aqddeg(x)&=0+1+W_{\ma q^2;q}(1)+3W_{\ma q^2;q}(2)
+3W_{\ma q^2;q}(3)\\[-.7em]
&\quad\;+5W_{\ma q^2;q}(4)+4W_{\ma q^2;q}(5)+3W_{\ma q^2;q}(6)
+W_{\ma q^2;q}(7)\\
&=\frac{(1-q^3)(1+2q+q^2+2q^3+q^5-\ma q^{10}-2\ma q^{12}-\ma q^{13}
-2\ma q^{14}-\ma q^{15})}{(1-q)(1-\ma q^3)}q^{-7},
\end{align*}
and
\begin{align*}
R(\la|\,\ma;q)&=
1+W_{\ma q^3;q}(1)+2W_{\ma q^3;q}(2)+3W_{\ma q^3;q}(3)
+3W_{\ma q^3;q}(4)+3W_{\ma q^3;q}(5)+W_{\ma q^3;q}(6)\\
&=\frac{(1-q^2)(1+2q+q^2+2q^3+q^5-\ma q^{10}-2\ma q^{12}
-\ma q^{13}-2\ma q^{14}-\ma q^{15})}{(1-q)(1-\ma q^4)}q^{-6}.
\end{align*}
Then according to \eqref{aqexp},
\begin{equation*}
\EE_{\ma;q}(X)=
\frac{\sum_{x\in [\emptyset,\lambda]} \aqddeg(x)}
{R(\la|\,\ma;q)}=
\frac{(1-q^3)(1-\ma q^4)}{(1-q^2)(1-\ma q^3)}q^{-1}
=\frac{[3]_{\ma q^2;q}}
{[2]_{\ma q^3;q}},
\end{equation*}
which agrees with the $(w,\ell,d)=(3,3,3)$ case of \eqref{aqexpf}.
\end{exam}

\section*{Acknowledgments}

The authors would like to thank Sam Hopkins for the alternative proof of
Theorem~\ref{thm:EX} in Remark~\ref{rem:antichain}.

Part of this article was written while the first author was participating in
the 2020 program in
Algebraic and Enumerative Combinatorics at Institut Mittag-Leffler. He
would like to thank the institute for the hospitality and Sara Billey, Petter
Br\"and\'en, Sylvie Corteel, and Svante Linusson for organizing the program.

This material is based upon work supported by the Swedish Research
Council under grant no. 2016-06596 while the first author was in residence
at Institut Mittag-Leffler in Djursholm, Sweden during the winter of 2020.

%----------------------------------------------------------------------------------------------------------

\end{document}